\documentclass{amsart}

\usepackage[english]{babel}
\usepackage[T1]{fontenc}
\usepackage{latexsym}
\usepackage{amsmath}
\usepackage{amsthm}
\usepackage{amssymb}
\usepackage{graphicx}
\usepackage{nicefrac}
\usepackage{amscd}
\usepackage{textcomp}
\usepackage{enumerate}
\usepackage{tikz}
\usetikzlibrary[arrows,decorations.pathmorphing,backgrounds,positioning,fit,petri]
\usepackage{pgf}
\usepackage{array}
\usepackage{xcolor}
\usepackage{mathrsfs}
\usepackage{float}
\usepackage{caption}
\usepackage{listings}
\usepackage{hyperref}
\usepackage{cancel}
\usepackage[titletoc,toc,title]{appendix}
\usepackage{imakeidx}

\DeclareMathOperator{\lip}{Lip}

\newcommand{\dd}{\, \mathrm{d}}
\newcommand{\aplim}{\operatorname*{aplim}}
\newcommand{\aplimsup}{\operatorname*{aplimsup}}

\newcommand{\bbR}{\mathbb{R}}

\date{\today}

\makeindex

\begin{document}

\title[$C^k$ Lusin Approximation for Functions $\bbG$ to $\mathbb{R}$]{A $C^k$ Lusin Approximation Theorem for Real-Valued Functions on Carnot Groups}

	\author[Marco Capolli]{Marco Capolli}
\address[Marco Capolli]{Institute of Mathematics, Polish Academy of Sciences, Jana i Jadrzeja Sniadeckich 8, Warsaw, 00-656, Poland}
\email[Marco Capolli]{Mcapolli@impan.pl}

\author[Andrea Pinamonti]{Andrea Pinamonti}
\address[Andrea Pinamonti]{Department of Mathematics, University of Trento, Via Sommarive 14, 38123 Povo (Trento), Italy}
\email[Andrea Pinamonti]{Andrea.Pinamonti@unitn.it}

\author[Gareth Speight]{Gareth Speight}
\address[Gareth Speight]{Department of Mathematical Sciences, University of Cincinnati, 2815 Commons Way, Cincinnati, OH 45221, United States}
\email[Gareth Speight]{Gareth.Speight@uc.edu}

\newtheorem{theorem}{Theorem}
\newtheorem{claim}[theorem]{Claim}
\newtheorem{corollary}[theorem]{Corollary}
\newtheorem{lemma}[theorem]{Lemma}
\newtheorem{proposition}[theorem]{Proposition}

\newtheorem*{theorem*}{Theorem}
\newtheorem*{lemma*}{Lemma}

\theoremstyle{definition}
\newtheorem{definition}[theorem]{Definition}
\newtheorem{assumptions}[theorem]{Assumptions}

\theoremstyle{remark}
\newtheorem{remark}[theorem]{Remark}
\newtheorem{example}[theorem]{Example}
\newtheorem{case}{Case}

\numberwithin{theorem}{section}
\numberwithin{equation}{section}

\newtheorem*{note*}{\underline{Important Note}}
\newenvironment{system}
  {\left\lbrace\begin{array}{*{6}{c@{}>{{}}c<{{}}@{}}c}}
  {\end{array}\right.}
\newcommand{\B}[1]{\mathbb{#1}}
\newcommand{\bbG}{\mathbb{G}}
\newcommand{\bbN}{\mathbb{N}}
\newcommand{\C}[1]{\mathcal{#1}}
\newcommand{\SF}[1]{\mathsf{#1}}
\newcommand{\SCR}[1]{\mathscr{#1}}
\newcommand{\defi}{\coloneqq}

\pagestyle{plain}

\pagestyle{headings}

\maketitle

\begin{abstract}
We study the Lusin approximation problem for real-valued measurable functions on Carnot groups. We prove that $k$-approximate differentiability almost everywhere is equivalent to admitting a Lusin approximation by $C^{k}_{\bbG}$ maps. We also prove that existence of an approximate $(k-1)$-Taylor polynomial almost everywhere is equivalent to admitting Lusin approximation by maps in a suitable Lipschitz function space.
\end{abstract}


\section{Introduction}

Nonsmooth maps arise frequently in analysis but can be challenging to work with. Hence it is useful to approximate them when possible by smooth maps. There are many ways to do this. For instance, Lusin's theorem asserts that given a measurable map $f\colon \mathbb{R}^{n}\to \mathbb{R}$ and $\varepsilon>0$, there exists a continuous map $F\colon \mathbb{R}^{n}\to \mathbb{R}$ such that $\mathcal{L}^{n}\{x\in\B{R}^n:F(x)\neq f(x)\}<\varepsilon$. Several refinements of Lusin's theorem show that the more regular the measurable function $f$, the more smooth the approximating function $F$ can be chosen. The best known result was proved by Federer \cite{Fed}. He showed that if a function $f:\mathbb{R}^n\to \mathbb{R}$ is almost everywhere differentiable, then for every $\varepsilon>0$ there exists a function $F\colon \mathbb{R}^{n}\to \mathbb{R}$ of class $C^{1}$ and a closed set $C\subset \mathbb{R}^n$ such that
$f=F$ on $C$ and $\mathcal{L}^{n}(\mathbb{R}^n\setminus C)<\varepsilon$.
Related results approximate absolutely continuous maps on $\mathbb{R}$ by $C^1$ maps, or $k$-approximately differentiable maps on $\mathbb{R}^{n}$ by $C^{k}$ maps \cite{LT94, Menne}. There are also approximation results for Sobolev maps \cite{Haj, Liu, MZim} and convex functions \cite{Aza1}. See also \cite{Isa}, \cite[Theorem 3.1]{DIKS}, \cite{Boj}. The present paper studies Lusin approximation for real-valued maps on Carnot groups which are $k$ approximately differentiable almost everywhere. Our main results (Theorem \ref{mainthma} and Theorem \ref{mainthmb}) partially extend the results of \cite{LT94} in the Euclidean setting. 

A large part of geometric analysis and geometric measure theory in Euclidean spaces may be generalized to more general settings. One particularly rich setting is that of Carnot groups \cite{BLU07, Mon02, Pan89}. Carnot groups are Lie groups whose Lie algebra admits a stratification. This stratification implies that points can be connected by horizontal curves. These are absolutely continuous curves with tangents in a distinguished subbundle of the tangent bundle. Considering lengths of horizontal curves yields the Carnot-Carath\'{e}odory distance. In addition to the group translations and distance, every Carnot group has a natural family of dilations and a Haar measure. This rich structure makes the study of analysis and geometry in Carnot groups appealing. However, results in the Carnot setting must respect the horizontal structure so can be quite different to the Euclidean setting.

One way to prove a Lusin approximation result is to apply a Whitney extension theorem. A Whitney extension theorem gives hypotheses under which a map defined on some subset can be extended to a smooth map on the whole space. This was first studied by Whitney in \cite{Whi34, Whi51}, but remains an active area of research, even in Euclidean spaces \cite{Feff2,Feff3}. To prove a Lusin approximation result, one typically starts with a nonsmooth map on the whole space, next deduces that it satisfies the hypotheses of a suitable Whitney extension theorem on some large compact set, then finally one applies the Whitney extension theorem to obtain the desired approximating smooth map. This strategy was implemented in the Euclidean context in \cite{LT94} to prove a $C^k$ Lusin approximation for maps on Euclidean spaces.

Lusin approximation in Carnot groups was first studied for horizontal curves in the Heisenberg group in \cite{S16}. The Heisenberg group is a Carnot group of step two and is the simplest non-Euclidean Carnot group. In \cite{S16} the third author showed directly that every horizontal curve coincides with a $C^{1}$ horizontal curve except for a set of small measure. However, the same result does not hold in the Engel group (a step three Carnot group). Independently, Zimmerman \cite{Zim18} proved a Whitney extension result for $C^{1}$ horizontal curves in the Heisenberg group. These results were extended by several authors to step two Carnot groups \cite{LS16}, pliable Carnot groups \cite{JS16} and sub-Riemannian manifolds \cite{SS18}.

Next, higher regularity results for horizontal curves were studied in the Heisenberg group. In \cite{PSZ19}, the second two authors and Zimmerman proved a Whitney extension result for $C^{k}$ horizontal curves in the Heisenberg group. In \cite{CPS21}, the authors of the present paper combine the results in \cite{PSZ19} with techniques from the Euclidean setting  \cite{LT94} to prove a $C^{k}$ Lusin approximation theorem for horizontal curves in the Heisenberg group. There the hypotheses on the nonsmooth maps required $(k-1)$-times $L^1$ differentiability of the first derivatives, rather than the weaker assumption of $k$-times approximate differentiability on the maps themselves which suffices in Euclidean spaces \cite{LT94}. More recently, Zimmerman investigated the Whitney finiteness principle for horizontal curves in the Heisenberg group \cite{Zim21}.

The results just mentioned focus on mappings from $\mathbb{R}$ into Carnot groups, i.e. curves. While mappings between general Carnot groups are at present out of reach, a $C^{k}$ Whitney extension theorem is known for $C^{k}$ mappings from a general Carnot group to $\mathbb{R}$ \cite{FSS01, FSC03, PV06}. In the present paper we prove a corresponding $C^k$ Lusin approximation result by combining techniques from \cite{LT94} and \cite{PV06}. Our main result is the following theorem. For the relevant definitions, see Section \ref{Background}.

\begin{theorem}\label{mainthma}
Let $D$ be a measurable subset of a Carnot group $\bbG$ and $f\colon D\to \B{R}$ be measurable. Then the following are equivalent for every non-negative integer $k$:
\begin{enumerate}
\item $f$ is $k$-approximately differentiable at almost every point of $D$. 
\item $f$ admits a Lusin approximation by functions in $C^{k}_{\bbG}(\bbG)$.
\end{enumerate} 
\end{theorem}

The hypothesis of approximate differentiability in Theorem \ref{mainthma} mirrors that of the Euclidean setting \cite{LT94}. This contrasts with the study of horizontal curves in the Heisenberg group \cite{CPS21}, where a stronger hypothesis of $L^1$ differentiability on the derivatives was necessary. However, in the Euclidean context, the conditions in Theorem \ref{mainthma} are equivalent to the seemingly weaker condition of having an approximate $(k-1)$-Taylor polynomial almost everywhere. This equivalence uses the fact that in Euclidean spaces $\mathrm{Lip}(k, \bbR^{n})$ maps admit a Lusin approximation by $C^k$ maps, which follows from \cite[Theorem 4]{Whi51}. Whether a similar fact holds in Carnot groups is unknown and will be the object of future investigation. This difference means that some steps in proving Theorem \ref{mainthma} are considerably more complicated, particularly proving the measurability of coefficients of $k$-approximate derivatives (Proposition \ref{diffmeas}).

Our second theorem investigates separately maps with approximate $(k-1)$ Taylor polynomials almost everywhere. We show such maps admit approximation by functions in $\mathrm{Lip}(k, \bbG)$. 
The space $\mathrm{Lip}(k, \bbG)$ has a complicated definition, but is a subspace of functions $u\in C^{k-1}_{\mathbb{G}}(\Omega)$ for which $X^{J}u$ is bounded for $|J|_{\bbG}\leq k-1$ and $X^{J}u$ is Lipschitz for $|J|_{\bbG}= k-1$ (Lemma \ref{mainthmb3}).

\begin{theorem}\label{mainthmb}
Let $D$ be a measurable subset of a Carnot group $\bbG$ with $\mathcal{L}^{N}~(D)~<~\infty$. Let $f\colon D\to \B{R}$ be measurable. Then the following are equivalent for every positive integer $k$:
\begin{enumerate}
\item $f$ has an approximate $(k-1)$-Taylor polynomial at almost every point of~$D$.
\item $f$ admits a Lusin approximation on $D$ by functions in $\mathrm{Lip}(k, \bbG)$.
\end{enumerate} 
\end{theorem}

Note that the assumption $\mathcal{L}^{N}(D)<\infty$ in Theorem \ref{mainthmb} is necessary since (1) is a local condition but (2) requires uniform bounds. For example, $f(x)=x^{2}$ on $\mathbb{R}$ has an approximate $(k-1)$-Taylor polynomial for any $k\geq 1$ but does not even admit a Lusin approximation by bounded functions on $\mathbb{R}$.

The proofs of Theorem \ref{mainthma} and Theorem \ref{mainthmb} adapt, in a nontrivial way, techniques from the Euclidean setting \cite{LT94} to the Carnot group setting and apply suitable generalizations of the Whitney extension theorem in Carnot groups \cite{PV06}. The main additional steps in the Carnot group setting involve proving uniqueness and measurability of the coefficients of the approximate derivatives and a suitable De Giorgi type lemma for polynomials (Lemma \ref{degiorgi_carnot}).

We now describe the organization of the paper. 

In Section \ref{Background} we recall the main background, including Carnot groups, Taylor polyomials, approximate derivatives, and the two extensions of Whitney's extension theorem to Carnot groups that we will apply.

In Section \ref{DiGiorgi} we first prove that approximate derivatives and approximate Taylor polynomials have uniquely determined coefficients at each point (Proposition \ref{structure_of_p} ). This justifies the definitions and is necessary since, to the best of our knowledge, they have not been studied before in the Carnot group setting. We then prove an analogue of the De Giorgi lemma for polynomials in Carnot groups (Lemma \ref{degiorgi_carnot}). Finally, we prove an estimate for the measure of intersections of balls (Lemma \ref{lemma_R}).

In Section \ref{measurability}, we prove that if a measurable map has an $k$-approximate derivative or $(k-1)$-approximate Taylor polynomial at almost every point $x$, then the coefficients of the derivative polynomials are measurable functions of $x$ (Proposition \ref{diffmeas} and Corollary \ref{taylormeas}). The proof is direct and significantly more complicated than the analogue in Euclidean spaces, since we did not have any way to approximate $\mathrm{Lip}(k, \bbG)$ maps by $C^{k}_{\bbG}(\bbG)$ maps to mirror the Euclidean argument. 

In Section \ref{Lusindiff}, we combine the results from Section \ref{DiGiorgi} and Section \ref{measurability} with Euclidean techniques and the classical Whitney extension theorem in Carnot groups (Theorem \ref{genclassicalWhitney}, which first appeared in \cite{PV06}) to prove our main result Theorem \ref{mainthma}. 

In Section \ref{Lusintaylor}, we follow similar arguments with a different Whitney extension theorem (Theorem \ref{extension}, which first appeared in \cite{PV06}) to prove Theorem \ref{mainthmb}.

\bigskip

\textbf{Acknowledgements:} The authors thank Professor Fon-Che Liu for helpful comments and references regarding the Euclidean analogues of these results, particularly concerning the approximation of $\mathrm{Lip}(k, \bbR^{n})$ functions by $C^k$ functions in Euclidean spaces. They also thank the anonymous referee for useful comments which improved the paper.

Part of this paper was written while M. Capolli was a PhD student at the University of Trento advised by the other two authors. A. Pinamonti is member of {\em Gruppo Nazionale per l'Analisi Ma\-te\-ma\-ti\-ca, la Probabilit\`a e le loro Applicazioni} (GNAMPA) of the {\em Istituto Nazionale di Alta Matematica} (INdAM). G. Speight  was  supported by a grant from the Simons Foundation (\#576219, G. Speight). Part of this paper was written while G. Speight was visiting the University of Trento and supported by funds from the University of Trento.

\section{Background}\label{Background}

\subsection{Carnot groups}

A \emph{Lie group} $\bbG$ is a smooth manifold which is also a group for which multiplication and inversion are smooth maps. The \emph{Lie algebra} $\mathfrak{g}$ associated to a Lie group is the space of left invariant vector fields equipped with the Lie bracket $[\cdot, \cdot]\colon \mathfrak{g}\times\mathfrak{g}\to \mathfrak{g}$. This is defined by
\[[X,Y](f)=X(Y(f))-Y(X(f))\qquad \mbox{for smooth }f\colon \bbG\to \mathbb{R}.\]
We denote the direct sum of vector spaces $V$ and $W$ by $V\oplus W$.

\begin{definition}\label{Carnotdef}
A simply connected Lie group $\bbG$ is said to be a \emph{Carnot group of step $s$} if its Lie algebra $\mathfrak{g}$ is stratified of step $s$. This means that there exist linear subspaces $V_1, \dots ,V_s$ of $\mathfrak{g}$ such that
\[\mathfrak{g}=V_1\oplus \dots \oplus V_s,\]
with
\[[V_1,V_{i}]=V_{i+1} \mbox{ if }1\leq i\leq s-1 \qquad \mbox{and} \qquad [V_1,V_s]=\{0\}.\]
Here $[V_1,V_i]:=\mathrm{span}\{[a,b]: a\in V_1,\ b\in V_i\}.$
\end{definition}

We fix throughout the paper a Carnot group $\bbG$ of step $s$ with Lie algebra $\mathfrak{g}$ admitting a stratification as in Definition \ref{Carnotdef}. Let $m_i:=\dim(V_i)$, $h_i:=m_1+\dots +m_i$ for $1\leq i\leq s$, and $h_{0}:=0$.  We define $m:=m_1=\dim(V_{1})$ and  $N:=h_s=\dim{\mathfrak{g}}$. A basis $X_1,\dots, X_N$ of $\mathfrak{g}$ is \emph{adapted to the stratification} if $X_{h_{i-1}+1},\dots, X_{h_{i}}$ is a basis of $V_i$ for $1\leq i \leq s$. We fix such a basis.

The map $\exp \colon \mathfrak{g}\to \mathbb{G}$ is defined by $\exp(X)=\gamma(1)$, where $\gamma \colon [0,1] \to \mathbb{G}$ is the unique solution to $\gamma'(t)=X(\gamma(t))$ and $\gamma(0)=e$, where $e$ is the identity element of $\mathbb{G}$.  This \emph{exponential map} is a diffeomorphism between $\mathbb{G}$ and $\mathfrak{g}$. We will identify $\mathbb{G}$ with $\mathbb{R}^{N}$ using the correspondence:
\[ \exp(x_{1}X_{1}+\dots +x_{N}X_{N})\in \mathbb{G} \longleftrightarrow (x_{1}, \dots, x_{N})\in \mathbb{R}^{N}.\]
With this identification, the identity element is $0\in \mathbb{R}^{N}$ and the inverse of $x$ is $-x$. We denote the product of $x, y\in\mathbb{G}$ by $xy$. The formula for $xy$ in coordinates will be given later in this section.

\subsection{Dilations, Haar Measure, and CC Distance}\label{section_dialtions}

We denote points of $\mathbb{G}$ by $(x_{1}, \dots, x_{N})\in \mathbb{R}^{N}$. The \emph{homogeneity} $d_i\in\bbN$ of the coordinate $x_i$ is defined by
\[ d_i:=j \quad\text {whenever}\quad h_{j-1}+1\leq i\leq h_{j}.\]
For any $\lambda >0$, the \emph{dilation} $\delta_\lambda\colon \bbG\to\bbG$, is defined in coordinates by
\[\delta_\lambda(x_1, \dots ,x_N)=(\lambda^{d_1}x_1, \dots ,\lambda^{d_N}x_N).\]
Dilations satisfy $\delta_{\lambda}(xy)=\delta_{\lambda}(x)\delta_{\lambda}(y)$ and $\left(\delta_{\lambda}(x)\right)^{-1}=\delta_{\lambda}(x^{-1})$.

A \emph{Haar measure} $\mu$ on $\mathbb{G}$ is a non-trivial Borel measure on $\mathbb{G}$ so that $\mu(gE)~=~\mu(E)$ for any $g\in \mathbb{G}$ and Borel set $E\subset \mathbb{G}$. Such a measure is unique up to scaling by a positive constant, so sets of measure zero are defined without ambiguity. In our identification of $\bbG$ with $\mathbb{R}^N$, any Haar measure is simply a constant multiple of $N$ dimensional Lebesgue measure $\mathcal{L}^{N}$. Our results are not sensitive to the choice of Haar measure, so we use $\mathcal{L}^N$ throghout. The terms measure and measurable will mean Lebesgue measure and Lebesgue measurable throughout the paper. 

Recall that a curve $\gamma\colon [a,b]\to \mathbb{R}^{N}$ is \emph{absolutely continuous} if it is differentiable almost everywhere, $\gamma' \in L^{1}[a,b]$, and $\gamma(t_{2})-\gamma(t_{1})=\int_{t_{1}}^{t_{2}} \gamma'(t) \dd t$ for all $t_{1}, t_{2}\in [a,b]$.

\begin{definition}\label{horizontalcurve}
An absolutely continuous curve $\gamma\colon [a,b]\to \mathbb{G}$ is \emph{horizontal} if there exist $u_{1}, \dots, u_{m}\in L^{1}[a,b]$ such that $\gamma'(t)=\sum_{j=1}^{m}u_{j}(t)X_{j}(\gamma(t))$ for almost every $t\in [a,b]$.
Define the \emph{horizontal length} of a horizontal curve $\gamma$ by $L(\gamma)=\int_{a}^{b}|u(t)|\dd t$, where $u=(u_{1}, \dots, u_{m})$ and $|\cdot|$ denotes the Euclidean norm on $\mathbb{R}^{m}$.
\end{definition}

The Chow-Rashevskii Theorem asserts that any two points of $\mathbb{G}$ can be connected by horizontal curves \cite[Theorem 9.1.3]{BLU07}. This allows us to define the \emph{Carnot-Carath\'eodory distance (CC distance)} as follows.

\begin{definition}
The \emph{CC distance} between $x,y\in \bbG$ is
\[d(x,y):=\inf \{ L(\gamma) : \gamma \colon [0,1]\to \mathbb{G} \mbox{ horizontal joining }x\mbox{ to }y \}.\]
\end{definition}

This CC distance satisfies $d(zx,zy)=d(x,y)$ and $d(\delta_{r}(x),\delta_{r}(y))=rd(x,y)$ for $x, y, z\in \mathbb{G}$ and $r>0$. We denote the CC ball with center $x$ and radius $r$ by $B(x,r)$. It can be proved that $\mathcal{L}^N(B(x,r))=r^Q \mathcal{L}^N(B(0,1))$, where $Q=\sum_{i=1}^s i m_i$. For the rest of the paper we denote $V:=\mathcal{L}^N(B(0,1))$.

The CC distance induces on $\mathbb{G}$ the same topology as the Euclidean distance, so we can refer unambiguously to open, closed, and compact sets. The CC distance is not bi-Lipschitz equivalent to the Euclidean distance. However, it can be proved that for any compact set $K\subset \mathbb{G}$ there are constants $c_1,c_2>0$ such that
\begin{equation}\label{disugdist}
c_1|x-y|\leq d(x,y)\leq c_2|x-y|^{\frac{1}{s}}\qquad \mbox{for all } x,y\in K.
\end{equation}
For convenience, we also denote $d(x,0)$ by $d(x)$ or $\|x\|$. The map $\|\cdot \| \colon \bbG\to \bbG$ is an example of a homogeneous norm, which are defined more generally as follows.
 
\begin{definition}
A function $\| \cdot\|_{\B{G}}:\mathbb{G}\to [0,\infty)$ is called a homogeneous norm if it satisfies both
\begin{enumerate}
\item $\|x\|_{\B{G}}>0$ if and only if $x\neq 0$,
\item $\|\delta_{r} x\|_{\B{G}}=r\|x\|_{\B{G}}$ for every $x\in \mathbb{G}$ and $r>0$.
\end{enumerate}
\end{definition}

The following is \cite[Proposition 5.1.4]{BLU07}. 

\begin{proposition}\label{equivalenza}
Let $\|\cdot\|_{\mathbb{G}}$ be a homogeneous norm on $\mathbb{G}$. Then there exists a constant $c>0$ such that
\[
c^{-1}\| x\|_{\mathbb{G}}\leq d(x)\leq c\|x\|_{\mathbb{G}} \quad \text{for all } x\in\mathbb{G}.
\]
\end{proposition}

\subsection{Polynomials and Smooth Functions}
In this section we give the relevant background on Polynomials and smooth functions on Carnot groups. We follow the notation introduced in \cite{FS}. A \emph{multi-index} $J=(j_1,\dots, j_N)$ is an ordered list of $N$ non-negative integers. Given a multi-index, we define its \emph{norm} as $|J| = \sum_{i=1}^N j_i$ and its \emph{homogeneous norm} as $|J|_{\B{G}} = \sum_{i=1}^N d_i j_i$. Whether $|\cdot|$ refers to the Euclidean norm on Euclidean space or the norm of a multi-index will be clear from the context. Finally $J!:=j_{1}!j_{2}!\dots j_{N}!$. The two norms are related by the inequality:
\begin{equation}\label{degree1}
|J|\leq |J|_{\mathbb{G}}\leq d_{N}|J| \mbox{ for every multi-index }J.
\end{equation}

We define homogeneous polynomials as follows \cite{PV06, BLU07}.

\begin{definition}
Given coordinates $x=(x_1,\dots,x_N)\in\mathbb{R}^N$, we define:
\begin{itemize}
\item A \emph{monomial of homogeneous degree $d\geq 0$} is a polynomial of the form $x^J:=x_1^{j_1}\dots x_N^{j_N}$ for some multi-index $J=(j_1,\dots, j_N)$ with $|J|_{\mathbb{G}}=d$.
\item A \emph{homogeneous polynomial of homogeneous degree $d$} is a linear combination of monomials of the same homogeneous degree $d$. 
\item A \emph{polynomial of homogeneous degree at most $d$} is a linear combination of monomials of homogeneous degree at most $d$.
\end{itemize}
\end{definition}

The following lemma follows easily from the definitions.

\begin{lemma}\label{polyfacts}
For every homogeneous polynomial $P$ of homogeneous degree $d$,
\begin{equation}\label{homogeneous}
P(\delta_{\lambda}x)=\lambda^{d} P(x).
\end{equation}  
If $P$, $Q$ are homogeneous polynomials of homogeneous degree $d_{1}$, $d_{2}$ respectively, then the product $PQ$ is a homogeneous polynomial of homogeneous degree $d_{1}+d_{2}$.
\end{lemma}

The explicit expression of $xy$ depends on  the Baker-Campbell-Hausdorff formula. It has the form
\begin{equation}\label{exprecoord}
xy=x+y+\mathcal{Q}(x,y)\quad \mbox{for all }x,y\in \mathbb{G}.
\end{equation}
Here $\mathcal{Q}=(\mathcal{Q}_1,\dots,\mathcal{Q}_N)\colon\mathbb{R}^N\times\mathbb{R}^N\to \mathbb{R}^N$ and each $\mathcal{Q}_i$ is a homogeneous polynomial of degree $d_i$. By \cite[(4)]{FSS}, $\mathcal{Q}_i$ is identically zero for $1\leq i\leq m$  and otherwise
\begin{equation}\label{defQi}
\mathcal{Q}_i(x,y)=\sum_{k,h} \mathcal{R}^i_{h,k}(x,y)(x_ky_h-x_hy_k), \quad \mbox{for}\ m<i\leq N.
\end{equation}
Here $\mathcal{R}^i_{h,k}$ are homogeneous polynomials of degree $d_i-d_k-d_h$ and the sum is extended over all $h, k$ such that $d_h+d_k \leq d_i$.

The following result will be helpful later. Recall that $|\cdot|$ denotes the Euclidean norm and $\bbG$ is identified with $\mathbb{R}^{m_1}\times\dots\times \mathbb{R}^{m_s}=\mathbb{R}^N$ via exponential coordinates. 

\begin{lemma}\label{dxj}
For every multi-index $J$, there exists $\tilde c>0$, depending only on $|J|_{\bbG}$, such that 
\[|x^J|\leq \tilde c d(x)^{|J|_{\mathbb{G}}}\qquad \mbox{for all }x\in\mathbb{G}.\]
\end{lemma}

\begin{proof} 
The following formula defines a homogeneous norm on $\bbG$ \cite{BLU07}:
\begin{equation*}
\| x\|_{\mathbb{G}}:=\sum_{j=1}^s |x^{(j)}|^{\frac{1}{j}},\quad x=(x^{(1)},\dots, x^{(s)})\in \mathbb{G},\ x^{(j)}\in \bbR^{m_{j}}.
\end{equation*}
By Proposition \ref{equivalenza}, there exists $c>0$ such that $\|x\|_{\mathbb{G}}\leq c d(x)$ for all $x\in\mathbb{G}$. Hence, replacing $c$ by $\max\{ c^j:1\leq j\leq s \}$, we obtain
\begin{equation*}
|x^{(j)}|\leq c d(x)^{j} \qquad \mbox{for all } j=1,\dots, s.
\end{equation*}
Writing $x=(x_{1},x_{2},\dots, x_{N})$ and recalling that $d_i$ is the homogeneity of  $x_i$,
\begin{equation}\label{1dc}
|x_i|\leq c d(x)^{d_i} \qquad \mbox{for all }i=1,\dots, N.
\end{equation}
Let $J=(j_1,\dots, j_N)$ be a multi-index and $\tilde c:=c^{|J|}$. Then by \eqref{1dc},
\[
|x^J|=|x_1^{j_1}\dots x_N^{j_N}|\leq \tilde c d(x)^{\sum_{i=1}^N j_i d_i}=\tilde c d(x)^{|J|_{\mathbb{G}}}.
\]
\end{proof}

Given a multi-index $J$, we denote higher order derivatives by $X^J=X_1^{j_1}\dots X_N^{j_N}$ and $D^{J}=\left( \frac{\partial}{\partial x} \right)^{J}=\frac{\partial^{j_{1}}}{\partial x_{1}^{j_{1}}}\dots \frac{\partial^{j_{N}}}{\partial x_{1}^{j_{N}}}$. The following lemma is \cite[Proposition 20.1.5]{BLU07}.

\begin{lemma}\label{esxa}
For every multi-index $\alpha$, there exist homogeneous polynomials $Q_{\beta,\alpha}$ of homogeneous degree $|\beta|_{\mathbb{G}}-|\alpha|_{\mathbb{G}}$ such that
\[
X^{\alpha}(x)=\left( \frac{\partial}{\partial x} \right)^{\alpha} + \sum_{\substack{\beta \neq \alpha \\ |\beta| \leq |\alpha|\\ |\beta|_{\mathbb{G}}\geq |\alpha|_{\mathbb{G}}}} Q_{\beta,\alpha}(x) \left(\frac{\partial}{\partial x}\right)^\beta.
\]
This equation is meant in the sense that both left and right define differential operators on $\mathbb{R}^{N}$ whose action agrees on $C^{\infty}$ functions on $\mathbb{R}^{N}$.
\end{lemma}

The following lemma follows from Lemma \ref{polyfacts}, Lemma \ref{esxa}, and the definitions.

\begin{lemma}\label{polyderiv}
Suppose $P$ is a homogeneous polynomial of homogeneous degree $k$ and $\alpha$ is a multi-index. Then $X^{\alpha}P$ is either identically zero or is a homogeneous polynomial of homogeneous degree $k-|\alpha|_{\bbG}$.
\end{lemma}

Given an open set $\Omega\subset \mathbb{G}$ and $k\in\mathbb{N}$, we define
\[C^k_{\mathbb{G}}(\Omega):=\left\{u: \Omega\to \mathbb{R} \colon X^J u \mbox{ exists and is continuous for all }  |J|_{\B{G}}\leq k\right\}. \]
Polynomials of the form $\sum_{|J|_{\mathbb{G}}\leq k} \alpha_J(x_0)\frac{(x_0^{-1}x)^J}{J!}$ are called \emph{polynomials centered at $x_0\in\mathbb{G}$}. Taylor polynomials in Carnot groups are defined as follows \cite{PV06}.

\begin{definition}
Let $\Omega\subset\B{G}$ be an open set and $k$ a non-negative integer. Let  $u\in C^k_{\mathbb{G}}(\Omega)$ and $x_0\in\Omega$. The \emph{Taylor polynomial of $u$ of homogeneous degree $k$ centered at $x_{0}$} is the unique polynomial $P_k(u,x_0,x)$ centered at $x_{0}$ with the property that
\begin{equation*}
(X^J u)(x_0)=X^J (P_k(u, x_0, x))|_{x=x_0}\qquad \mbox{for all }|J|_{\mathbb{G}}\leq k.
\end{equation*}
\end{definition}

The following results are well known, see for instance \cite[Theorem 1]{ACC} and \cite[Corollary 1.44]{FS}. They motivate the definition of approximate Taylor polynomial and approximate differentiability to be given in Definition \ref{apTaylorapDiif}.

\begin{theorem} \label{Teo_Taylor}
Fix $k\in\mathbb{N}$, $\Omega\subset\mathbb{G}$ open, and  $u\in C^k_{\mathbb{G}}(\Omega)$. Then for every $x_0\in \Omega$
\[\limsup_{x\to x_0}\frac{|u(x)-P_{k-1}(u,x_0,x)|}{d(x,x_0)^k}<\infty\]
and
\[ \lim_{x\to x_0}\frac{|u(x)-P_k(u,x_0,x)|}{d(x,x_0)^k}=0. \]
\end{theorem}

The following fact is \cite[Proposition 3]{PV06}.

\begin{proposition}\label{dxa}
Let $k$ be a non-negative integer and $P$ a polynomial of homogeneous degree $k$. Then 
\[P(x)=P_k(P,x_0,x) \qquad \mbox{for any }x_0\in\mathbb{G}.\]
\end{proposition}

The following lemma is attributed in \cite{CA64} to De Giorgi. We will prove an analogue for polynomials in Carnot groups in Lemma \ref{degiorgi_carnot}.

\begin{lemma}[De Giorgi]\label{DeGiorgi}
Let $E\subset\B{R}^n$ be a measurable subset of a ball $B(x_0,r)$ such that $\mathcal{L}^n(E)\geq Ar^n$ for some constant $A>0$. Then for each positive integer $k$ there exists a positive constant $C$, depending only on $n,k$ and $A$, such that
\[
|D^{\alpha}P(x_0)|\leq \frac{C}{r^{n+|\alpha|}}\int_E |P(y)| \dd y
\]
for all polynomials $P$ of degree at most $k$ and for all multi-indices $\alpha$.
\end{lemma}

\subsection{Approximate Derivatives and Approximate Taylor Polynomials}

We say that a measurable set $D\subset \bbG$ has density one at a point $x\in \bbG$ if
\[\lim_{r\to 0^{+}} \frac{\mathcal{L}^{N}(B(x,r)\cap D)}{\mathcal{L}^{N}(B(x,r))}=1.\]
We will also say $x$ is a density point of $D$ to mean the same thing. We say that $D$ has density zero at $x\in \bbG$ if $\bbG\setminus D$ has density one at $x$. Almost every point $x\in D$ of a measurable set $D\subset \bbG$ is a density point of $D$. This follows from the fact that $(\bbG, d, \mathcal{L}^{N})$ is a doubling metric measure space and the Lebesgue differentiation theorem holds in all such spaces.

Let $f\colon D \to \bbR$ be a measurable function. We write $\aplim_{y\to x} f(y)=l$ if the set $\{y\in D: |f(y)-l|\leq \varepsilon\}$ has density one at $x$ for any $\varepsilon>0$. We say that $f$ is approximately continuous at a point $x\in D$ if $\aplim_{y\to x}f(y)=f(x)$. It follows from the previous paragraph that every measurable function is approximately continuous at almost every point in its domain.

We denote by $\aplimsup_{y\to x} f(y)$ the infimum of all $\lambda\in\mathbb{R}$ such that the set $\{y\in D: f(y)>\lambda\}$ has density zero at $x$.

\begin{definition}\label{apTaylorapDiif}
Let $D\subset \bbG$ be measurable, $f\colon D\to \bbR$ be measurable, and $x_{0}\in D$ be a point at which $D$ has density one.

We say that $f$ has an \emph{approximate $(k-1)$-Taylor polynomial at $x_0$} for some positive integer $k$ if there is a polynomial $p(x_0,x)$ centred at $x_0$ of homogeneous degree at most $(k-1)$ with
\begin{equation}\label{approxTaylor}
\aplimsup_{x\to x_0}\frac{|f(x)-p(x_0,x)|}{d(x,x_0)^k}<\infty.
\end{equation}

We say that $f$ is \emph{approximately differentiable of order $k$ at $x_0$} for some non-negative integer $k$ if there is a polynomial $p(x_0,x)$ centred at $x_0$ of homogeneous degree at most $k$ with
\begin{equation}\label{apprxDiff}
\aplim_{x\to x_0}\frac{|f(x)-p(x_0,x)|}{d(x,x_0)^k}=0.
\end{equation}
\end{definition}

In Proposition \ref{structure_of_p}  we will verify that the approximate $(k-1)$-Taylor polynomial and approximate derivative of order $k$ are unique at any point where they exist.

\begin{lemma}\label{approxtaylor}
Let $D\subset \bbG$ be measurable, $f\colon D\to \bbR$ be measurable, and $x_{0}\in D$ be a point of density of $D$. If $f$ has an approximate $(k-1)$-Taylor polynomial at $x_0$, then it is approximately differentiable of order $(k-1)$ at $x_0$ with the same approximate Taylor polynomial.
\end{lemma}

\begin{proof}
This follows from
\begin{equation*}
\aplimsup_{x\to x_0}\frac{|f(x)-p(x_0,x)|}{d(x,x_0)^{k-1}}=\aplimsup_{x\to x_0}\frac{|f(x)-p(x_0,x)|}{d(x,x_0)^k}d(x,x_0)=0.
\end{equation*}
\end{proof}

\subsection{Whitney Extension and Lusin Approximation}

Let $\Omega\subset \bbG$ be an open set, $u\in C^k_{\mathbb{G}}(\Omega)$, and $x_{0}\in \bbG$. The Taylor polynomial of $u$ can be written as
\begin{equation}\label{Pk}
P_k(u,x_0,x):=\sum_{|J|_{\mathbb{G}}\leq k} \alpha_J(x_0)\frac{(x_0^{-1}x)^J}{J!}
\end{equation}
for some coefficients $\alpha_J(x_0)$. By \cite[equation (2.9) on page 603]{PV06}, there exist constants $\beta_{JK}$, independent of $u$ and $x_{0}$, such that for $|J|_{\bbG}\leq k$,
\begin{align}\label{def_a_J}
\alpha_J(x_0)&=(X^J u)(x_0)+\sum_{\substack{   |K|_{\mathbb{G}}=|J|_{\mathbb{G}} \\ |K|<|J|}} \beta_{JK} (X^K u)(x_0) \\
\nonumber
&=\sum_{\substack{   |K|_{\mathbb{G}}=|J|_{\mathbb{G}} \\ |K|\leq |J|}} \beta_{JK} (X^K u)(x_0).
\end{align}
The  $\beta_{JK}$ satisfy, among other properties, $\beta_{II}=1$ and $\beta_{IK}=0$ if $|K|_{\mathbb{G}}=|I|_{\mathbb{G}}$ with $|K|=|I|$ and $K\neq I$. 

Suppose $u\colon \Omega\to \bbR$ is any function for which $X^{J}f(x)$ exist for all $|J|_{\bbG}\leq k$. Then we will also denote by $P_{k}(u,x_{0},x)$ the polynomial in \eqref{Pk} with coefficients given by \eqref{def_a_J}. If $u\in C^{k}_{\bbG}(\Omega)$, this agrees with the Taylor polynomial given earlier.

%

Suppose $\{f_J\}_{|J|_{\mathbb{G}}\leq k}$ is a collection of real-valued functions on a closed set $D\subset \bbG$. Then for every $x_{0}\in D$ we denote by $P_{k}(\{f_J\}_{|J|_{\mathbb{G}}\leq k} ,x_{0},x)$ the polynomial in \eqref{Pk} with coefficients given by \eqref{def_a_J} with $X^{K}u(x_{0})$ replaced by $f_{K}(x_{0})$. 

The following is the extension to Carnot groups of the classical Whitney extension theorem \cite[Theorem 4]{PV06}. It will be used in the proof of Theorem \ref{mainthma}.

\begin{theorem}\label{genclassicalWhitney}
Let $k\in\B{N}$ and $\{f_{J}\}_{|J|_{\bbG}\leq k}$ be a collection of functions on a closed set $F\subset \bbG$ satisfying the following conditions:
\begin{itemize}
\item $|f_{J}(x)|\leq M$ for $|J|_{\bbG}\leq k$ on every compact subset of $F$.
\item $R_J(x,y):=f_J(y)-X^J (P_{k}(\{f_J\}_{|J|_{\mathbb{G}}\leq k} ,x,y))$ is $o(d(x,y)^{k-|J|_{\bbG}})$ for every $|J|_{\bbG}\leq k$ in the following sense. For every $\varepsilon>0$ and $\bar{x}\in F$, there is $\delta=\delta(\varepsilon,\bar{x})>0$ such that
\[|R_{J}(x,y)|\leq \varepsilon d(x,y)^{k-|J|_{\bbG}}\]
for all $x,y\in F$ satisfying $d(\bar{x},x)<\delta$ and $d(\bar{x},y)<\delta$ and all $|J|_{\bbG}\leq k$. 
\end{itemize}
Then there exists $f\in C^{k}_{\bbG}(\bbG)$ such that $X^{J}f|_{F}=f_{J}$ for all $|J|_{\bbG}\leq k$.
\end{theorem}

We now recall some definitions that will be useful to state the Lipschitz version of the Whitney extension theorem.

\begin{definition}\label{lipG}
Fix $k\in\mathbb{N}$, $k<\gamma\leq k+1$. A function $f\colon \bbG\to \bbR$ belongs to $\lip(\gamma,\bbG)$ if $X^{J}f(x)$ exists for all $x\in \bbG$ and $|J|_{\bbG}\leq k$, and there is a constant $M$ such that  for all $|J|_{\mathbb{G}}\leq k$,
\begin{equation}\label{defLipG}
|X^{J}f(x_0)|\leq M,\qquad |R_J(x_0,x)|\leq Md(x,x_0)^{\gamma-|J|_{\mathbb{G}}}\ \mbox{for all}\ x,x_0\in \bbG,
\end{equation}
where 
\[
R_J(x_0,x):=(X^{J}f)(x)-X^J(P_{k}(f,x_0,x))
\]
\end{definition}

\begin{definition}\label{lipD}
Fix $k\in\mathbb{N}$, $k<\gamma\leq k+1$, and let $D\subset \mathbb{G}$ be a closed set. A collection $\{f_J\}_{|J|_{\mathbb{G}}\leq k}$ of real-valued functions on $D$ belongs to $\lip(\gamma,D)$ if there is a constant $M$ such that  for all multi-indices $|J|_{\mathbb{G}}\leq k$
\begin{equation}\label{defLipD}
|f_J(x_0)|\leq M,\qquad |R_J(x_0,x)|\leq Md(x,x_0)^{\gamma-|J|_{\mathbb{G}}}\ \mbox{for all}\ x,x_0\in D,
\end{equation}
where 
\[
R_J(x_0,x):=f_J(x)-X^J (P_{k}(\{f_J\}_{|J|_{\mathbb{G}}\leq k} ,x_{0},x)).
\]
\end{definition}

We will make use of the space $\lip(k,D)$ and $\lip(k,\bbG)$. Notice that these are obtained from the previous definitions by replacing $k$ by $k-1$ then $\gamma$ by $k$.

In \cite[Remark 1]{PV06} the authors observe that if $D=\mathbb{G}$, then the two definitions of $\lip(\gamma,D)$ are consistent. More precisely, the following lemma holds.

\begin{lemma}\label{LipEquiv}
Fix $k\in\mathbb{N}$ and let $f:\mathbb{G}\to \mathbb{R}$ be such that there is family $\{f_J\}_{|J|_{\mathbb{G}}\leq k}$ of real-valued functions on $\mathbb{G}$ with $f_J = f$ when $|J|=0$ and $\{f_J\}_{|J|_{\mathbb{G}}\leq k} \in\lip(\gamma,\B{G})$ for all $|J|_{\B{G}}\leq k$. Then for $|J|_{\mathbb{G}}\leq k$, $f$ has continuous derivatives $X^Jf$ satisfying $f_J=X^J f$ and \eqref{defLipG} holds for some, possibly different, constant $M$.
\end{lemma}

We now state the Lipschitz version of the Whitney extension theorem in Carnot groups \cite[Theorem 2 and Lemma 6]{PV06}. It will be used in the proof of Theorem \ref{mainthmb}.

\begin{theorem}\label{extension}
Let $k\in\mathbb{N}$, $k<\gamma\leq k+1$, and let $D\subset \mathbb{G}$ be a closed set. Let $\{f_J\}|_{J|_{\mathbb{G}}\leq k} \in \lip(\gamma,D)$. Then there exists a function $f\in  \lip(\gamma,\bbG)$ such that $X^{J}f(x)=f_{J}(x)$ for all $x\in D$ for all $|J|_{\bbG} \leq k$.
\end{theorem}

We now define the Lusin approximation property we will study.

\begin{definition}
Let $D\subset \mathbb{G}$ and $f\colon D\to \bbR$ be measurable. Let $k\in\mathbb{N}$. We say that $f$ has the \emph{Lusin property of order $k$} or \emph{admits a Lusin approximation by functions in $C^{k}_{\bbG}(\bbG)$} if for every $\varepsilon>0$ there exists $u\in C^k_{\mathbb{G}}(\mathbb{G})$ such that
\[
\mathcal{L}^N \{x\in D:u(x)\neq f(x)\}<\varepsilon.
\]
If a similar statement holds with $C^k_{\mathbb{G}}(\mathbb{G})$ replaced by $\lip(k,\bbG)$, then we say $f$ \emph{admits a Lusin approximation by functions in $\lip(k,\bbG)$}.
\end{definition}

\section{Polynomials and the De Giorgi Lemma in Carnot Groups}\label{DiGiorgi}

In this section we show that the coefficients of approximate derivatives and approximate Taylor polynomials at a point are unique, prove a version of the De Giorgi lemma in Carnot groups, and prove a simple estimate for the measure of intersection of balls.

\subsection{Uniqueness of Approximate Taylor Polynomials and Approximate Derivatives} 

We first prove a simple lemma, showing how a point of density can be approached from many directions within the corresponding set. The distance of a point $x\in \bbG$ from a set $A\subset \bbG$ will be denoted by $d(x,A)=\inf \{d(x,y): y\in A\}$. Recall that $V$ is the measure of the unit ball with respect to the CC metric.

\begin{lemma}\label{densityclaim}
Fix $v\in \bbG$ with $d(v)=1$ and $0<\theta<1$. Let $\mathcal{N}(v)=\{\delta_{t}v:t>0\}$ and define the set
\[\mathcal{C}(v,\theta)=\{w\in \bbG: d(w,\mathcal{N}(v))<\theta d(w)\}.\]
Then $\mathcal{C}(v,\theta)$ has positive lower density at $0$, i.e. 
\[\liminf_{R \downarrow 0} \frac{\mathcal{L}^{N}(B(0,R)\cap \mathcal{C}(v,\theta))}{R^Q} >0.\]
Consequently, if $A\subset \bbG$ is a measurable set with density one at $0$, then we have $B(0,R)\cap \mathcal{C}(v,\theta)\cap A\neq \varnothing$ for all sufficiently small $R>0$.
\end{lemma}

\begin{proof}
Fix $R>0$. We will show that if $L=\frac{\theta}{4(1+\theta)}$, then
\begin{equation}\label{inclusion_balls}
B(\delta_{R/2}v,LR)\subset B(0,R)\cap \mathcal{C}(v,\theta).
\end{equation}
To this end, suppose $y\in B(\delta_{R/2}v,LR)$. Then $d(y,\delta_{R/2}v)<LR$. By definition of $L$ and the fact $0<\theta<1$, it follows $L<\theta/4\theta=1/4$. Hence
\[d(y)\leq d(y,\delta_{R/2}v)+d(\delta_{R/2}v)<R/2+R/2=R.\]
Hence $y\in B(0,R)$. To verify $y\in \mathcal{C}(v,\theta)$, it suffices to verify $d(y,\delta_{R/2}v)<\theta d(y)$. By the triangle inequality, we have $d(y)\geq d(\delta_{R/2}v)-LR=R/2-LR$. Since $d(y,\delta_{R/2}v)<LR$ from the definition of $y$, it suffices to have $LR<\theta(R/2-LR)$. Equivalently $L<\theta/2-\theta L$ or $L(1+\theta)<\theta/2$. This is valid for our choice of $L$, so $y\in \mathcal{C}(v,\theta)$ follows. This verifies \eqref{inclusion_balls}. Hence
\[\mathcal{L}^{N}(B(0,R)\cap \mathcal{C}(v,\theta))\geq V L^{Q}R^{Q}.\]
Since $V$ and $L$ are independent of $R$, this proves the first part of the claim. The second part is an easy consequence, since the density of $A$ in balls $B(0,R)$ approaches one as $R\downarrow 0$.
\end{proof}

The following two lemmas are adaptations of \cite[Lemma 4.1]{Del} and \cite[Proposition 4.1]{Del} from the Euclidean setting.

\begin{lemma}\label{homlimit}
Let $P \colon \mathbb{G} \to \mathbb{R}$  be a homogeneous polynomial. Let $A\subset \mathbb{G}$ be measurable and $x\in \mathbb{G}$ such that $A$ has density one at $x$. Suppose $\varphi \colon A\setminus \{x\}\to \mathbb{R}$ is defined by
\[
\varphi(y)=P\left(\delta_{\frac{1}{d(x,y)}}(x^{-1}y)\right)
\]
satisfies
\begin{equation}\label{zerolimit}
\aplim_{\substack{y\to x\\ y\in A}} \varphi(y)=0.
\end{equation}
Then $P$ is the constant zero polynomial.
\end{lemma}

\begin{proof}
For integer $h\geq 1$, let
\begin{equation*}
A_h:=\left\{y\in A\setminus\{x\} : |\varphi(y)|< 1/h\right\}.
\end{equation*}

Equation \eqref{zerolimit} implies that $x$ is a point of density one of $A_{h}$ for all $h>1$. Fix $v\in \bbG$ with $d(v)=1$ and recall the set $\mathcal{N}(v)$. Apply Lemma \ref{densityclaim} with $A$ replaced by $x^{-1}A_{h}$, $v$ as fixed, $\theta=1/h$, and apply translation by $x$. This shows that, for $h\geq 1$, there exist
\[y_h\in A_h\cap x\mathcal{C}(v,1/h)\cap B(x,1/h).\]

Next define
\begin{equation*}
z_h:=\delta_{\frac{1}{d(x,y_h)}}(x^{-1}y_h).
\end{equation*}
We claim that $z_{h}\to v$ as $h\to \infty$. The definition of $y_{h}$ implies $x^{-1}y_{h}\in \mathcal{C}(v,1/h)$. Hence there exists $t>0$ such that $d(x^{-1}y_{h},\delta_{t}v)<d(x^{-1}y_{h})/h$. Applying dilations, $d(z_{h},\delta_{\frac{t}{d(x^{-1}y_{h})}}v)<1/h$. By the triangle inequality, it follows
\[ |d(z_{h})-d(\delta_{\frac{t}{d(x,y_{h})}}v)|<1/h.\]
Hence
\[|1-\frac{t}{d(x,y_{h})}|<1/h.\]
This implies $\frac{t}{d(x^{-1}y_{h})}\to 1$ and so, by continuity of dilations, $\delta_{\frac{t}{d(x^{-1}y_{h})}}v \to v$. Then
\[d(z_{h},v)\leq d(z_{h},\delta_{\frac{t}{d(x^{-1}y_{h})}}v) + d(\delta_{\frac{t}{d(x^{-1}y_{h})}}v,v)\to 0.\]
Hence $z_{h}\to v$ as claimed.

Finally, $y_{h}\in A_{h}$ implies $|\varphi(y_{h})|<1/h$. By definition of $\varphi$ in terms of $P$, it follows $|P(z_{h})|<1/h$. Using continuity of $P$ and $z_{h}\to v$ yields $P(v)=0$.
The conclusion follows from the arbitrary choice of $v\in \partial B(0,1)$ and the fact that $P$ is homogeneous. 
\end{proof}

\begin{proposition}[Uniqueness]\label{structure_of_p} 
Suppose $f\colon D\to \bbR$ is a measurable function defined on a measurable set $D\subset \bbG$. Let $x_{0}\in D$ be a point of density of $D$.

\begin{enumerate}
\item Suppose $f$ has an approximate $(k-1)$-Taylor polynomial at $x_{0}$ for some positive integer $k$, denoted
\[p(x_0,x)=\sum_{|J|_{\mathbb{G}}\leq k-1} \alpha_J(x_0)\frac{(x_0^{-1}x)^J}{J!}.\]
Then the coefficients $\alpha_J(x_0)$, $|J|_{\mathbb{G}}\leq k-1$, are uniquely determined. Hence $p(x_0,x)$ is uniquely determined.
\item Suppose $f$ is approximately differentiable of order $k$ at $x_0$ for some non-negative integer $k$, with polynomial
\[p(x_0,x)=\sum_{|J|_{\mathbb{G}}\leq k} \alpha_J(x_0)\frac{(x_0^{-1}x)^J}{J!}.\]
Then the coefficients $\alpha_J(x_0)$, $|J|_{\mathbb{G}}\leq k$, are uniquely determined. Hence $p(x_0,x)$ is uniquely determined.
\end{enumerate}
\end{proposition}

\begin{proof}
We start by proving (1). Suppose that $q(x_0,x)$ is another polynomial centred at $x_0$ of homogeneous degree at most $(k-1)$ for which \eqref{approxTaylor} holds with $p(x_0,x)$ replaced by $q(x_0,x)$. Then
\[
\aplimsup_{x \to x_{0}}\frac{|p(x_0,x)-q(x_0,x)|}{d(x_0,x)^k}<\infty.
\]
While polynomomials are defined on $\B{G}$, note that all approximate limits in this proof are over the domain $D$ of $f$. Since the polynomials $p$ and $q$ are centered at $x_{0}$, we can write
\begin{align*}
&p(x_0,x)=\sum_{|J|_{\mathbb{G}}\leq k-1} \alpha_J(x_0)\frac{(x_0^{-1}x)^{J}}{J!},\\
&q(x_0,x)=\sum_{|J|_{\mathbb{G}}\leq k-1} \beta_J(x_0)\frac{(x_0^{-1}x)^{J}}{J!}.
\end{align*}
We have
\begin{align}\label{polato}
0&=\aplim_{x\to x_{0}}\frac{|p(x_0,x)-q(x_0,x)|}{d(x_0,x)^{k-1}}\\
\nonumber
&=\left|\aplim_{x\to x_{0}}\sum_{|J|_{\mathbb{G}}\leq k-1}\frac{(\alpha_J(x_0)-\beta_J(x_0))}{J!}\frac{(x_0^{-1}x)^J}{d(x_0,x)^{k-1}}\right|.
\end{align}
Multiplying by $d(x_{0},x)^{k-1}$ and taking the limit gives
\begin{equation}\label{Q0}
\alpha_{\underline{0}}(x_0)=\beta_{\underline{0}}(x_0)
\end{equation}
where $\underline{0}$ denotes the multi-index with all entries $0$.

Next for $0\leq i\leq k-1$ define
\[
Q_i(z):=\sum_{|J|_{\mathbb{G}}=i}\frac{(\alpha_J(x_0)-\beta_J(x_0))}{J!}z^J.
\]
Notice that each $Q_{i}$ is a homogeneous polynomial of homogeneous degree $i$. By Lemma \ref{dxj}, there exists $C>0$ such that 
\begin{equation}\label{stimaqw}
|Q_i(z)|\leq C d(z)^i \quad \mbox{for all}\ z\in\mathbb{G}, \quad i=1,\dots, k-1.
\end{equation}

Using \eqref{Q0} and the definition of $Q_{i}$, the terms in \eqref{polato} can be written as
\[
\left|\aplim_{x\to x_0}\sum_{i=1}^{k-1} \frac{Q_i(x_0^{-1}x)}{d(x_0,x)^{k-1}}\right|=0.
\]

We now verify by induction that $Q_i \equiv 0$ for all $0\leq i\leq k-1$.  By \eqref{Q0} we have $Q_0\equiv 0$. Assume $Q_0 \equiv \dots \equiv Q_h \equiv 0$, with $0\leq h \leq k-2$. Then 
\[\left|\aplim_{x\to x_0}\sum_{i=h+1}^{k-1} \frac{Q_i(x_0^{-1}x)}{d(x_0,x)^{k-1}}\right|=0.\]
Hence
\[\left|\aplim_{x\to x_0}\sum_{i=h+1}^{k-1} \frac{Q_i(x_0^{-1}x)}{d(x_0,x)^{h+1}}\right|=0.\]
It follows from \eqref{stimaqw} that
\[ \aplim_{x\to x_{0}} \frac{|Q_{h+1}(x_0^{-1}x)|}{d(x_0,x)^{h+1}}=0.\]
The proof of (1) then follows by applying Lemma \eqref{homlimit} and using the fact that $Q_{h+1}$ is homogeneous of degree $h+1$.

The proof of (2) is similar, except starting at \eqref{polato} with $k-1$ replaced by $k$.
%
 \end{proof}

\subsection{De Giorgi Lemma in Carnot Groups}

\begin{lemma}[De Giorgi Lemma for Carnot groups]\label{degiorgi_carnot}
Let $E\subset \B{G}$ be a measurable subset of a ball $B(x_0,r)$ such that $\mathcal{L}^N(E)\geq Ar^Q$ for some $A>0$. Let $k$ be a positive integer.

Then there exists a positive constant $C$, depending only on $k,Q$ and $A$, such that for all polynomials $P$ of homogeneous degree at most $k$ and multi-indices $\alpha$,
\[
|X^{\alpha}P(x_0)|\leq \frac{C}{r^{Q+|\alpha|_{\B{G}}}}\int_E |P(y)| \dd y.
\]
\end{lemma}

\begin{proof}
By Lemma \ref{polyderiv}, $X^{\alpha} P\equiv 0$ for every multi-index $\alpha$ with $|\alpha|_{\mathbb{G}}\geq k+1$. Hence to prove the inequality we can assume without loss of generality $|\alpha|_{\mathbb{G}}< k+1$.

We first prove the lemma in the case $r=1$ and $x_{0}=0$. Let $E$ be a measurable subset of $B(0,1)$ such that $\mathcal{L}^N(E)\geq A$. By \eqref{disugdist}, $B(0,1)\subset B_\mathbb{E}(0,c_1^{-1})$, where $B_\mathbb{E}(0,c_1^{-1})$ denotes the Euclidean ball centred at $0$ with radius $c_1^{-1}$.
By \eqref{degree1}, any polynomial $W$ of homogeneous degree at most $k$ has also standard degree at most $k$. Since $W$ is a polynomial, it is a $C^{\infty}$ function on $\bbR^{N}$ so we can apply Lemma \ref{esxa} to compute its derivatives. Recall the polynomials $Q_{\beta,\alpha}$ from Lemma \ref{esxa}. 
Let $C_{1}$ be the constant from Lemma \ref{DeGiorgi}. Combining Lemma \ref{DeGiorgi} with Lemma \ref{esxa}, for every multi-index $\alpha$ with $|\alpha|_{\bbG}<k+1$,
\begin{align}\label{stima0}
|(X^{\alpha} W)(0)|  & \leq \sum_{\substack{|\beta|\leq |\alpha|\\ |\beta|_{\mathbb{G}}\geq |\alpha|_{\mathbb{G}}}} |Q_{\beta,\alpha}(0)| \left(\left(\frac{\partial}{\partial x}\right)^\beta W\right)(0)\\
\nonumber
&\leq C_{1} \sum_{\substack{|\beta|\leq |\alpha|\\ |\beta|_{\mathbb{G}}\geq |\alpha|_{\mathbb{G}}}} |Q_{\beta,\alpha}(0)| \int_{E} |W(y)| \dd y
\end{align}
This proves the lemma in the case $r=1$ and $x_{0}=0$.

We now prove the lemma for general $r>0$ and $x_{0}\in \bbG$. Take a measurable subset $E$ of $B(x_0,r)$ such that $\mathcal{L}^N(E)\geq  A r^{Q}$. Let $W$ be a polynomial of homogeneous degree at most $k$. Denote by $T:\mathbb{G}\to\mathbb{G}$ the map $T(x)=\delta_{\frac{1}{r}}(x_0^{-1}x)$. Hence $T^{-1}(x)=x_0\delta_{r}(x)$ and the classical change of variable formula gives:
\begin{equation}\label{cambio}
\int_{E}|W(y)| \dd y = r^Q\int_{T(E)}|W(x_0\delta_{r} y)| \dd y.
\end{equation}
Further, $E\subset B(x_{0},r)$ implies
\begin{equation}\label{cambio1}
T(E)\subset B(0,1)
\end{equation}
and $\mathcal{L}^N(E)\geq A r^{Q}$ yields
\begin{equation}\label{stima1}
\mathcal{L}^N(T(E))\geq \frac{1}{r^Q}\mathcal{L}^N(E)\geq A.
\end{equation}
Define $S(y):=W(x_0\delta_r y)$. This is a homogeneous polynomial of degree at most $k$ in the variable $y$. Using \eqref{stima0} with $E$ replaced by $T(E)$, $W$ replaced by $S$, \eqref{cambio}, \eqref{cambio1} and \eqref{stima1},
\begin{align}
|(X^{\alpha} S)(0)|\leq C \int_{T(E)} |S(y)|\, dy=\frac{C}{r^Q}\int_E |W(y)| \dd y.
\end{align}
To conclude the proof it suffices to prove that
\begin{equation}\label{tesifinaledegiorgi}
(X^{\alpha} S)(0)=r^{|\alpha|_{\mathbb{G}}}(X^{\alpha}W)(x_{0}) \qquad \mbox{for every multi-index }\alpha.
\end{equation}
Clearly both sides are linear in $W$. Hence we can assume without loss of generality that $W$ is a homogeneous polynomial of homogeneous degree $m$. By Lemma \ref{polyderiv}, $X^{\alpha} W$ is either identically $0$ or is a homogeneous polynomial of homogeneous degree $m-|\alpha|_{\mathbb{G}}$. In either case, $(X^{\alpha}W)(\delta_{\lambda}x)=\lambda^{m-|\alpha|_{\mathbb{G}}}(X^{\alpha}W)(x)$. Let $L_{g}(x)=gx$ denote left translation by $g$. Since $X^{\alpha}$ is left invariant for any multi-index $\alpha$, we know $X^{\alpha}(\varphi \circ L_{g})=(X^{\alpha}\varphi)\circ L_{g}$ for any smooth function $\varphi\colon \bbG\to \bbR$ and $g\in \bbG$. Then
\begin{align*}
(X^{\alpha} S)(y) &=X^{\alpha}(W(x_0\delta_r y))\\
&=r^m X^\alpha(W(\delta_{\frac{1}{r}}(x_0)y))\\
&=r^m X^{\alpha}((W\circ L_{\delta_{\frac{1}{r}}x_{0}})(y))\\
&=r^{m}(X^{\alpha}W)(L_{\delta_{\frac{1}{r}}x_{0}}y)\\
&=r^m r^{|\alpha|_{\mathbb{G}}-m} (X^{\alpha}W)(x_{0}\delta_{r}y)\\
&=r^{|\alpha|_{\mathbb{G}}} (X^{\alpha}W)(x_{0}\delta_{r}y).
\end{align*}
Substituting $y=0$ yields \eqref{tesifinaledegiorgi} and completes the proof.
\end{proof}

\subsection{Measures of Intersections of Balls}

The following simple lemma will be useful in the proof of Theorem \ref{mainthma} and Theorem \ref{mainthmb}. Recall that $V$ denotes the measure of the unit ball with respect to the CC metric.

\begin{lemma}\label{lemma_R}
For every $x, y\in \bbG$,
\[
\frac{V}{2^{Q}} \leq\frac{\mathcal{L}^N(B(x,d(x,y))\cap B(y,d(x,y)))}{d(x,y)^Q}\leq V.
\]
\end{lemma}
 
\begin{proof}
Fix $x,y\in \bbG$ and denote $\delta=d(x,y)$. Clearly
\[ \mathcal{L}^N(B(x,\delta)\cap B(y,\delta)) \leq \mathcal{L}^{N}(B(x,\delta))=V\delta^{Q},\]
which gives the upper bound.

Since $(\bbG, d)$ is a geodesic metric space, we can fix a point $z\in \bbG$ such that $d(x,z)=d(y,z)=\delta/2$. We claim that 
\begin{equation}\label{inclusion}
B(z,\delta/2)\subset B(x,\delta)\cap B(y,\delta).
\end{equation}
Indeed, suppose $w\in B(z,\delta/2)$. Then
\[d(w,x)\leq d(w,z)+d(z,x)<\delta/2 + \delta/2=\delta.\]
This shows $w\in B(x,\delta)$. A similar argument shows $w\in B(y,\delta)$, proving \eqref{inclusion}. Using \eqref{inclusion} then gives
\[ \mathcal{L}^N(B(x,\delta)\cap B(y,\delta)) \geq \mathcal{L}^{N}(B(z,\delta/2))=V\delta^{Q}/2^{Q},\]
which gives the lower bound and completes the proof.
\end{proof}

\section{Approximate Derivatives and Approximate Taylor Polynomials Have Measurable Coefficiets}\label{measurability}

In this section we prove that the coefficients of approximate derivatives and approximate Taylor polynomials are measurable. This will be important in the proof of Theorem \ref{mainthma} and Theorem \ref{mainthmb}.

\subsection{Distance Estimates}

We first verify two estimates that will be used in the proof of measurability.

\begin{lemma}\label{convergedistance}
For every $x_{0}\in \bbG$, $r>0$, and $k\in\mathbb{N}$,
\[\sup_{x\in B(y_0,r)} |d(x,x_0)^k - d(x,y_0)^k   |\to 0 \mbox{ as }y_{0}\to x_{0}.\]
\end{lemma}

\begin{proof}
We use induction on $k$. For $k=1$, the lemma follows from the triangle inequality. Suppose the lemma is true up to $k-1$. For the $k$'th case, assume $d(y_{0},x_{0})<1$. Then for $x\in B(y_0,r)$,
\begin{align*}
\big|& d(x,x_0)^k - d(x,y_0)^k \big|\\
&=\big| d(x,x_0)^k-d(x,x_0)^{k-1}d(x,y_0)+d(x,x_0)^{k-1}d(x,y_0)-d(x,y_0)^k \big|\\
&\leq d(x,x_0)^{k-1}\big|d(x,x_0)-d(x,y_0)\big|+d(x,y_0)\big| d(x,x_0)^{k-1}-d(x,y_0)^{k-1} \big|\\
&\leq \left(r+1\right)^{k-1}\big|d(x,x_0)-d(x,y_0)\big|+r\big| d(x,x_0)^{k-1}-d(x,y_0)^{k-1} \big|.
\end{align*}
The proof is concluded by taking the supremum and using the inductive step.
\end{proof}

\begin{lemma}\label{convergeindex}
For every $x_{0}\in \bbG$, $r>0$, and multi-index $J$,
\[
\sup_{x\in B(y_0,r)} |(x_0^{-1}x)^J-(y_0^{-1}x)^J|\to 0 \mbox{ as }y_{0}\to x_{0}.
\]
\end{lemma}

\begin{proof}
Recall $\mathcal{Q}_i(x,y)$ and $\mathcal{R}^i_{h,k}(x,y)$ defined in \eqref{exprecoord} and \eqref{defQi}. We first prove the following claim.

\begin{claim}\label{stimaQ}
For every $x_{0}\in \bbG$, $r>0$ and $1\leq i\leq N$
\[
\sup_{x\in B(y_0,r)} |\mathcal{Q}_i(x_0,x)-\mathcal{Q}_i(y_0,x)|\to 0 \mbox{ as }y_{0}\to x_{0}.
\]
\end{claim}

\begin{proof}[Proof of Claim \ref{stimaQ}]
By \eqref{defQi}, we have
\begin{align*}
&|\mathcal{Q}_i(x_0,x)-\mathcal{Q}_i(y_0,x)|\\
\nonumber
&\qquad \leq \sum_{h,k} |\mathcal{R}^i_{h,k}(x_0,x)||x_k(x_{0,h}-y_{0,h})-x_h(x_{0,k}-y_{0,k})|\\
\nonumber
&\qquad \qquad +\sum_{h,k}|x_ky_{0,h}-x_hy_{0,k}||\mathcal{R}^i_{h,k}(x_0,x)-\mathcal{R}^i_{h,k}(y_0,x)|.
\end{align*}
Recalling that each $\mathcal{R}^i_{h,k}$ is a homogeneous polynomial, we can write 
\begin{align}\label{exr}
\mathcal{R}^i_{h,k}(x_0,x)=\sum_{j} d^{i}_{h,k,j} x^{\alpha_j}x_{0}^{\beta_{j}}
\end{align}
for suitable $d^{i}_{h,k,j}\in\mathbb{R}$ and multi-indices $\alpha_j=(\alpha_{j1},\dots, \alpha_{jN}),\beta_j=(\beta_{j1},\dots,\beta_{jN})$. Recall $\tilde{c}$ from Lemma \ref{dxj}. Combining Lemma \ref{dxj} with \eqref{exr}, we have
\begin{align}\label{prima}
&|\mathcal{R}^i_{h,k}(x_0,x)-\mathcal{R}^i_{h,k}(y_0,x)|\\
&\qquad \leq \sum_{j}| d^{i}_{h,k,j}| |x^{\alpha_{j}}||x_{0}^{\beta_{j}}-y_{0}^{\beta_{j}}| \nonumber \\
\nonumber
&\qquad \leq \tilde c\sum_{j}| d^{i}_{h,k,j}| d(x)^{|\alpha_j|_{\mathbb{G}}}|x_{0}^{\beta_{j}}-y_{0}^{\beta_{j}}|.
\end{align}
Also,
\begin{align}\label{prima2}
|\mathcal{R}^i_{h,k}(x_0,x)|&\leq \sum_{j}| d^{i}_{h,k,j}| |x^{\alpha_{j}}|x_{0}^{\beta_{j}}|\\
\nonumber
&\leq \tilde c\sum_{j}| d^{i}_{h,k,j}| d(x)^{|\alpha_j|_{\mathbb{G}}} |x_{0}^{\beta_{j}}|.
\end{align}
By the triangle inequality,
\begin{align}\label{seconda}
d(x)&\leq d(x,y_0)+d(x_0,y_0)+d(x_0)\\
\nonumber
&\leq r+d(x_0,y_0)+d(x_0).
\end{align}
Assume $d(x_{0},y_{0})<1$. Recall $c_{1}>0$ from \eqref{disugdist} with $K=\overline{B}(x_0,r+1)$ and let $C=1/c_{1}$. For any $h,k$,
\begin{align}\label{terza}
|x_ky_{0,h}-x_hy_{0,k}|&= |y_{0,h}(x_{k}-y_{0,k})+y_{0,k}(y_{0,h}-x_{h})|\\
&\leq 2|x-y_0||y_0| \nonumber\\
\nonumber
&\leq 2C d(x,y_0)|y_0|
\end{align}
and
\begin{align}\label{terza2}
|x_k(x_{0,h}-y_{0,h})-x_h(x_{0,k}-y_{0,k})|&\leq 2(|x-y_0|+|y_0|)|x_0-y_0|\\
\nonumber
&\leq 2C(Cd(x,y_0)+|y_0|)d(x_0,y_0).
\end{align}
Combining \eqref{prima},\eqref{seconda} and \eqref{terza}, we easily conclude 
\begin{equation}\label{55}
\sup_{x\in B(y_0,r)}\sum_{h,k}|x_ky_{0,h}-x_hy_{0,k}||\mathcal{R}^i_{h,k}(x_0,x)-\mathcal{R}^i_{h,k}(y_0,x)|\to 0\ \mbox{as}\ y_0\to x_0.
\end{equation}
Similarly using \eqref{prima2}, \eqref{seconda} and \eqref{terza2}, we get
\begin{equation}\label{56}
\sup_{x\in B(y_0,r)} \sum_{h,k} |\mathcal{R}^i_{h,k}(x_0,x)||x_k(x_{0,h}-y_{0,h})-x_h(y_{0,k}-x_{0,k})|\to 0\ \mbox{as}\ y_0\to x_0.
\end{equation}
Finally, combining \eqref{55} and \eqref{56} proves the claim.
\end{proof}

To prove the lemma, we use induction on the length $|J|_{\bbG}$ of $J$. If $|J|_{\B{G}}=0$ the result is clear. If $|J|_{\bbG}=1$, then $J=e_i$ for some $1\leq i\leq m$ where $m$ is the dimension of the first layer. By \eqref{exprecoord},
\[
|(x_0^{-1}x)^J-(y_0^{-1}x)^J| = |(x_i-x_{0,i})-(x_i-y_{0,i})|=|y_{0,i}-x_{0,i}|
\]
so the result is clear.

Before proving the induction step, we record one more estimate. Let $\overline C=2\tilde c C$. Let $1\leq i\leq N$ and $x\in B(y_{0},r)$. Using \eqref{defQi}, \eqref{prima2}, \eqref{terza} and \eqref{seconda}, we get
\begin{align}\label{789}
|\mathcal{Q}_i(-x_0,x)|&\leq \sum_{h,k} |\mathcal{R}_{h,k}^i(-x_0,x)||x_{0,h}x_k-x_{0,k}x_h|\\
\nonumber
&\leq \overline{C} \sum_{h,k}\sum_j |d_{h,k,j}^i| d(x)^{|\alpha_j|_{\mathbb{G}}}|x_0^{\beta_j}||x_0|d(x,x_0)\\
\nonumber
&\leq \overline{C} \sum_{h,k}\sum_j |d_{h,k,j}^i| (r+1+d(x_0))^{|\alpha_j|_{\mathbb{G}}}|x_0^{\beta_j}||x_0|(r+1).
\end{align}

Next we prove that if the lemma is true for every multi-index $J$ with $|J|_{\B{G}}\leq k-1$, then it also holds for every multi-index $J$ with $|J|_{\B{G}}=k$. Fix $J=(j_1,\dots, j_N)$ with $|J|_{\B{G}}=k$. Choose $i$ such that $j_i\neq 0$ and define $J'=J - e_i$. Clearly $|J'|_{\B{G}}=k-d_i$ where $d_i$ is the homogeneity of $e_i$ as defined at the beginning of Section \ref{section_dialtions}.
For every $x\in B(y_0,r)$, assuming $d(x_{0},y_{0})<1$,
\begin{align*}\label{567}
		&\big|(x_0^{-1}x)^J-(y_0^{-1}x)^J\big|\\
		& =\big| (x_0^{-1}x)^{J'}(x_{i}-x_{0,i}+\mathcal{Q}_i(-x_0,x))-(y_0^{-1}x)^{J'}(x_i-y_{0,i}+\mathcal{Q}_i(-y_0,x)) \big|\\
		& \leq
		\big|  (x_0^{-1}x)^{J'}(x_{i}-x_{0,i}+\mathcal{Q}_i(-x_0,x)) -  (y_0^{-1}x)^{J'}(x_{i}-x_{0,i}+\mathcal{Q}_i(-x_0,x))|\\
		&\qquad +| (y_0^{-1}x)^{J'}(x_{i}-x_{0,i}+\mathcal{Q}_i(-x_0,x)) - (y_0^{-1}x)^{J'}(x_{i}-y_{0,i}+\mathcal{Q}_i(-y_0,x)) \big|\\
		& \leq |x_{i}-x_{0,i}+\mathcal{Q}_i(-x_0,x)|\big| (x_0^{-1}x)^{J'}-(y_0^{-1}x)^{J'} \big| \\
		&\qquad + |(y_0^{-1}x)^{J'}|\big| y_{0,i}-x_{0,i}-\mathcal{Q}_i(-y_0,x)+\mathcal{Q}_i(-x_0,x)\big|\\
		& \leq Cd(x,x_0)\big| (x_0^{-1}x)^{J'}-(y_0^{-1}x)^{J'} \big| + \tilde c d(x,y_0)^{|J'|_{\mathbb{G}}}\big| y_{0,i}-x_{0,i}-\mathcal{Q}_i(-y_0,x)+\mathcal{Q}_i(-x_0,x)\big|\\
		&\qquad+|\mathcal{Q}_i(-x_0,x)|\big| (x_0^{-1}x)^{J'}-(y_0^{-1}x)^{J'} \big| \\
		&\leq C(r+1)\big| (x_0^{-1}x)^{J'}-(y_0^{-1}x)^{J'} \big|+\tilde c Cr^{|J'|_{\mathbb{G}}}d(x_0,y_0)+\tilde c r^{|J'|_{\mathbb{G}}}\big|\mathcal{Q}_i(-y_0,x)-\mathcal{Q}_i(-x_0,x)\big|\\
		&\qquad +|\mathcal{Q}_i(-x_0,x)|\big| (x_0^{-1}x)^{J'}-(y_0^{-1}x)^{J'} \big|.
		\end{align*}
The conclusion follows by combining the above estimate with the inductive assumption, Claim \ref{stimaQ}, \eqref{789}, and the fact that $y_0\to x_0$.
\end{proof}

\subsection{Statement and Reduction to Approximate Derivatives on $\bbG$}

The following proposition is the main result of this section. To avoid repeatedly writing factorials, we will denote $k$-approximate derivatives by $\sum_{|J|_{\mathbb{G}}\leq k} \alpha_J(x_0)(x_0^{-1}x)^J$ in this section only. Since this amounts to changing the coefficients $\alpha_J$ by a factor $J!$, clearly it does not affect whether they are measurable.

\begin{proposition}\label{diffmeas}
Let $D$ be a measurable subset of $\mathbb{G}$ and $k$ be a non-negative integer. Suppose $f\colon D\to \mathbb{R}$ is measurable and approximately differentiable of order $k$ at almost every point of $D$ with associated polynomial $\sum_{|J|_{\mathbb{G}}\leq k} \alpha_J(x_0)(x_0^{-1}x)^J$. Then the coefficients $\alpha_{J}$ are measurable functions on $D$ for $|J|_{\mathbb{G}}\leq k$.
\end{proposition}

Before proving Proposition \ref{diffmeas}, we observe how it easily implies that coefficients of approximate Taylor polynomials are also measurable.

\begin{corollary}\label{taylormeas}
Let $D$ be a measurable subset of $\mathbb{G}$ and $k$ is a positive integer. Suppose $f\colon D\to \mathbb{R}$ is measurable and has an approximate $(k-1)$-Taylor polynomial at almost every point of $D$, denoted by $\sum_{|J|_{\mathbb{G}}\leq k-1} \alpha_J(x_0)(x_0^{-1}x)^J$. Then the coefficients $\alpha_{J}$ are measurable functions on $D$ for $|J|_{\mathbb{G}}\leq k-1$.
\end{corollary}

\begin{proof}[Proof of Corollary \ref{taylormeas} from Proposition \ref{diffmeas}]
Suppose $f\colon D\to \mathbb{R}$ is measurable with approximate $(k-1)$-Taylor polynomial $p(x_{0},x)=\sum_{|J|_{\mathbb{G}}\leq k-1} \alpha_J(x_0)(x_0^{-1}x)^J$ at almost every $x_{0}\in D$. By Lemma \ref{approxtaylor}, $f$ is approximately differentiable of order $k-1$ at almost every $x_{0}\in D$ with polynomial $p(x_{0},x)$ of homogeneous degree at most $k-1$. By Proposition \ref{diffmeas}, it follows that the coefficients $\alpha_{J}$ are measurable on $D$ for $|J|_{\mathbb{G}}\leq k-1$.
\end{proof}

We use the rest of this section to prove Proposition \ref{diffmeas}.

%
%

\subsection{Reduction to Bounded Functions with Domain $\bbG$}

\begin{lemma}
Suppose Proposition \ref{diffmeas} holds in the special case $D=\bbG$. Then it holds for general measurable $D\subset \bbG$.
\end{lemma}

\begin{proof}
Suppose Proposition \ref{diffmeas} holds in the case when the domain is $\bbG$. Let $D\subset \bbG$ be an arbitrary measurable set. Let $f\colon D \to \bbR$ be measurable and $k$-approximately differentiable at almost every point of $D$ with polynomial $p(x,y)$ in $y$ at almost every point $x\in D$. Define $F\colon \bbG\to \bbR$ by $F|_{D}=f$ and $F$ identically zero on $G\setminus D$. Clearly $F$ is measurable. The map $F$ is $k$-approximately differentiable almost everywhere on $\bbG$. If we denote the associated polynomial at a point $x$ by $q(x,y)$, then $q(x,y)=p(x,y)$ and the coefficients agree for almost every $x\in D$. By Proposition \ref{diffmeas} in the case when the domain is $\bbG$, the coefficients of $q(x,y)$ are measurable on $\bbG$. Hence the coefficients of $p(x,y)$ are measurable on $D$.
\end{proof}

From now on we work with measurable functions whose domain is $\bbG$. The following easy lemma will be useful in reducing the proof of Proposition \ref{diffmeas} to a simpler case.

\begin{lemma}\label{consistent}
Suppose $f,g\colon \bbG\to \bbR$ are measurable and $x_{0}$ is a density point of $\{x\in \bbG:f(x)=g(x)\}$. Suppose $f$ is $k$-approximately differentiable at $x_{0}$ with $k$-approximate derivative $p(x_{0},x)$. Then $g$ is also $k$-approximately differentiable at $x_{0}$ with the same $k$-approximate derivative $p(x_{0},x)$.
\end{lemma}

\begin{proof}
Clearly
\begin{equation}
\begin{split}
\aplim_{x\to x_0}\frac{|g(x)-p(x_0,x)|}{d(x,x_0)^k}&\leq \aplim_{x\to x_0}\frac{|f(x)-p(x_0,x)|}{d(x,x_0)^k}+\aplim_{x\to x_0}\frac{|g(x)-f(x)|}{d(x,x_0)^k}\\
&=\aplim_{x\to x_0}\frac{|g(x)-f(x)|}{d(x,x_0)^k}.
\end{split}
\end{equation}
Since $x_0$ is a density point of $\{x\in \bbG:f(x)=g(x)\}$ we have
\begin{equation}
\aplim_{x\to x_0}\frac{|g(x)-f(x)|}{d(x,x_0)^k}=0.
\end{equation}
This proves the lemma.
\end{proof}

We next show that we can assume $f$ is bounded. Given $f\colon \bbG\to \bbR$ and $h\in \bbN$, define $f_{h}(x)=f(x)$ if $-h<f(x)<h$, $f_{h}(x)=h$ if $f(x)\geq h$, $f_{h}(x)=-h$ if $f(x)\leq -h$. Clearly $f_{h}$ is measurable whenever $f$ is measurable. The following lemma follows from Lemma \ref{consistent}.

\begin{lemma}
Suppose $f\colon \bbG\to \bbR$ is measurable and $k$-approximately differentiable almost everywhere. Then $f_{h}$ is $k$-approximately differentiable almost everywhere for every $h\in \bbN$. The coefficients of the $k$-approximate derivatives of $f_{h}$ and $f$ respectively satisfy $a_{J}^{h}(x_{0})=a_{J}(x_{0})$ for $|J|\leq k$ whenever $x_{0}$ is a density point of $\{x\in \bbR^{n}: -h<f(x)<h\}$.
\end{lemma}

The following lemma shows it suffices to prove Proposition \ref{diffmeas} for bounded $f$, in addition to already assuming the domain is $\bbG$.

\begin{lemma}\label{boundedisok}
Suppose $f\colon \bbG\to \bbR$ is measurable and $k$-approximately differentiable almost everywhere. Assume the coefficients $a_{J}^{h}$ of the $k$-approximate derivatives of $f_{h}$ are measurable for $|J|\leq k$ and every $h\in \bbN$. Then the coefficients $a_{J}$ of the $k$-approximate derivatives of $f$ are measurable for $|J|\leq k$.
\end{lemma}

\begin{proof}
Let $A_{h}=\{x\in \bbG: -h<f(x)<h\}$ and let $B_{h}$ be the set of density points of $A_{h}$. Clearly $\cup_{m=1}^{\infty}A_{h}=\bbG$ and $\mathcal{L}^{N}(A_{h}\setminus B_{h})=0$ for all $h$ by the Lebesgue density theorem. Define $S=\cup_{h=1}^{\infty}B_{h}$. Then
$\mathcal{L}^{N}(\bbG\setminus S)=0.$

Let $x_{0}\in S$. Then $x_{0}\in B_{T}$ for some $T\geq 1$. Since $A_{h}$ and hence $B_{h}$ are increasing sequences of sets, it follows that $x_{0}\in B_{h}$ for all $h\geq T$. By the previous lemma, $a_{J}^{h}(x_{0})=a_{J}(x_{0})$ for all $|J|\leq k$ and $h\geq T$. In particular, $a_{J}^{h}(x_{0})\to a_{J}(x_{0})$ for $|J|\leq k$. Since this holds for any $x_{0}\in S$ and $\mathcal{L}^{N}(\bbG\setminus S)=0$, it follows $a_{J}^{h}\to a_{J}$ almsot everywhere for $|J|\leq k$. Hence $a_{J}$ are measurable for $|J|\leq k$.
\end{proof}

\subsection{Set Decomposition}

We now prove Proposition \ref{diffmeas} assuming $D=\bbG$ and $f$ is bounded. We first encode the conditions for the coefficients of the $k$ derivatives to lie inside an interval in terms of simpler sets with countable quantifiers.
 
\begin{lemma}\label{decomposition}
	Suppose measurable $f\colon \B{G}\to\B{R}$ is $k$-approximately differentiable at a point $x_0\in\B{G}$ with $k$-approximate derivative $\sum_{|J|_{\mathbb{G}}\leq k}a_J(x_0)(x_0^{-1}x)^J$. Let $B_J\subset\B{R}$, $|J|_{\mathbb{G}}= k$, be non-empty closed bounded intervals. Then the following are equivalent:
	\begin{enumerate}
		\item $a_J(x_0)\in B_J$ for all $|J|_{\mathbb{G}}= k$.
		\item For each $\varepsilon \in\B{Q}^+$, there exist $R\in\B{Q}^+$ and $q_J\in B_J\cap \B{Q}$ for $|J|_{\mathbb{G}}=k$ such that for all $r\in (0,R)\cap\B{Q}$, 
		\begin{align*}
		&\mathcal{L}^{N} \Big\{ x\in B(x_0,r): \frac{|f(x)-\sum_{|J|_{\mathbb{G}}\leq k-1}a_J(x_0)(x_0^{-1}x)^J-\sum_{|J|_{\mathbb{G}}= k}q_J(x_0^{-1}x)^J|}{d(x,x_0)^k}>\varepsilon \Big\}\\
		&\qquad <\varepsilon r^Q.
		\end{align*}
	\end{enumerate}
\end{lemma}

\begin{proof}
Let $T$ be the number of multi-indices $J$ such that $|J|_{\mathbb{G}}=k$. 

We first prove $(1)\Rightarrow (2)$. Fix $\varepsilon\in\B{Q}^+$. For each $J$ such that $|J|_{\mathbb{G}}=k$, choose $q_J\in B_J\cap \B{Q}$ such that $|q_{J}-a_J(x_0)|< \varepsilon / (2\tilde cT)$, where $\tilde c>0$ is as in Lemma \ref{dxj}. Using the definition of $k$-approximate differentiability, making $R$ slightly smaller if necessary to make it rational, there exists $R\in\B{Q}^+$ such that the following holds. For all $0<r<R$, and so in particular for all $r\in(0,R)\cap\B{Q}$,
\begin{equation}\label{measure_f_minus_approx}
\mathcal{L}^{N} \Big\{ x\in B(x_0,r): \frac{|f(x)-\sum_{|J|_{\mathbb{G}}\leq k}a_J(x_0)(x_0^{-1}x)^J|}{d(x,x_0)^k}>\varepsilon / 2 \Big\} <\varepsilon r^Q.
\end{equation}
However,
\begin{align*}
&\frac{|\sum_{|J|_{\mathbb{G}}=k}a_J(x_0)(x_0^{-1}x)^J - \sum_{|J|_{\mathbb{G}}=k}q_J(x_0^{-1}x)^J|}{d(x,x_0)^k}\\
&\qquad \leq \sum_{|J|_{\mathbb{G}}=k}|a_J(x_0)-q_J|\frac{|(x_0^{-1}x)^J|}{d(x,x_0)^k}\\
&\qquad \leq \tilde cT\cdot \frac{\varepsilon}{2\tilde cT}\frac{d(x,x_0)^k}{d(x,x_0)^k}\\
&\qquad = \varepsilon/2.
\end{align*}
Combining this with \eqref{measure_f_minus_approx}, we obtain
\begin{equation*}
\begin{split}
&\mathcal{L}^{N} \Big\{ x\in B(x_0,r): \frac{|f(x)-\sum_{|J|_{\mathbb{G}}\leq k-1}a_J(x_0)(x_0^{-1}x)^J-\sum_{|J|_{\mathbb{G}}= k}q_J(x_0^{-1}x)^J|}{d(x,x_0)^k}>\varepsilon \Big\}\\
&\qquad \leq  \mathcal{L}^{N} \Big\{ x\in B(x_0,r): \frac{|f(x)-\sum_{|J|_{\mathbb{G}}\leq k}a_J(x_0)(x_0^{-1}x)^J|}{d(x,x_0)^k}> \varepsilon/2 \Big\} \\
&\qquad < \varepsilon r^Q
\end{split}
\end{equation*}
This proves $(1)\Rightarrow (2)$.

To prove $(2) \Rightarrow (1)$, we begin by applying condition $(2)$ with $\varepsilon = 1/n$ for $n\in\B{N}$. This gives constants $R_n\in\B{Q}^+$ and $q_J^n\in B_J\cap\B{Q}$ for $|J|_{\mathbb{G}}=k$ such that for all $r\in (0,R_n)\cap\B{Q},$
\begin{align*}
&\mathcal{L}^{N} \Big\{ x\in B(x_0,r): \frac{|f(x)-\sum_{|J|_{\mathbb{G}}\leq k-1}a_J(x_0)(x_0^{-1}x)^J-\sum_{|J|_{\mathbb{G}}= k}q_J^n(x_0^{-1}x)^J|}{d(x,x_0)^k}>1/n \Big\}\\
&\qquad  < r^Q/n.
\end{align*}
Both sides of the above inequality are continuous in $r$. Hence the inequality holds for all $r\in(0,R_n)$, provided the strict bound by $r^{Q}/n$ is replaced by a weak inequality. For every $|J|_{\mathbb{G}}=k$, $B_J$ is compact so the sequence $\{q_J^n\}_{n\in\B{N}}$ admits a convergent subsequence. Recall $\tilde{c}$ from Lemma \ref{dxj}. Replacing $q_{J}^{n}$ with a subsequence if necessary, we can assume that $q_J^n\to q_J\in B_J$ for each $|J|_{\mathbb{G}}=k$. For $0<r<R_n$ we have,
\begin{align*}
&\mathcal{L}^{N} \Big\{ x\in B(x_0,r):\frac{|f(x)-\sum_{|J|_{\mathbb{G}}\leq k-1}a_J(x_0)(x_0^{-1}x)^J-\sum_{|J|_{\mathbb{G}}= k}q_J(x_0^{-1}x)^J|}{d(x,x_0)^k}>\frac{1}{n}+\tilde{c}T|q_J^n-q_J| \Big\}\\
&\qquad \leq \mathcal{L}^{N} \Big\{ x\in B(x_0,r):\frac{|f(x)-\sum_{|J|_{\mathbb{G}}\leq k-1}a_J(x_0)(x_0^{-1}x)^J-\sum_{|J|_{\mathbb{G}}= k}q_J^n(x_0^{-1}x)^J|}{d(x,x_0)^k}>\frac{1}{n} \Big\}\\ 
&\qquad < r^Q / n.
\end{align*}
Given $\varepsilon>0$, choose $n\in\B{N}$ such that $\frac{1}{n}+\tilde{c}T|q_J^n-q_J|<\varepsilon$. For  $0<r<R_n$,
\begin{align*}
		&\mathcal{L}^{N} \Big\{ x\in B(x_0,r):\frac{|f(x)-\sum_{|J|_{\mathbb{G}}\leq k-1}a_J(x_0)(x_0^{-1}x)^J-\sum_{|J|_{\mathbb{G}}= k}q_J(x_0^{-1}x)^J|}{d(x,x_0)^k}>\varepsilon\Big\}\\
		&\qquad  <r^Q / n < \varepsilon r^Q.
\end{align*}
This shows $f$ is $k$-times approximately differentiable at $x_0$ with $a_J(x_0)=q_{J}\in B_J$ for all $|J|_{\mathbb{G}}=k$.
\end{proof}

\subsection{Proof of Measurability}

We prove Proposition \ref{diffmeas} by induction on $k$. By Lemma \ref{boundedisok}, we can assume that $f$ is bounded and $D=\B{G}$. We start with $k=0$.

\begin{lemma}
Suppose a bounded measurable function $f\colon \B{G}\to \B{R}$ is $0$-approximately differentiable almost everywhere. Then the coefficients of the $0$-approximate derivative are measurable.
\end{lemma}

\begin{proof}
The $0$-approximate derivative at each point $x_{0}$ where it exists is a constant depending on $x_{0}$. Denote it by $A(x_{0})$. The definition of $0$-approximate differentiability gives $\aplim_{x\to x_{0}} |f(x)-A(x_{0})|=0$ for almost every $x_{0}$. In other words, $\aplim_{x\to x_{0}}f(x)=A(x_{0})$ for almost every $x_{0}$. However, every measurable function is approximately continuous almost everywhere, so $\aplim_{x\to x_{0}}f(x)=f(x_{0})$ for almost every $x_{0}$. Hence $f(x_{0})=A(x_{0})$ for almost every $x_{0}$. Since $f$ is measurable, it follows $A$ is measurable.
\end{proof}

We now establish the induction step.

\begin{lemma}\label{induction}
Fix positive integer $k$ so that whenever a bounded measurable function $f\colon \B{G}\to \B{R}$ is $(k-1)$-approximately differentiable almost everywhere, it follows necessarily that the coefficients $a_{J}$, $|J|_{\mathbb{G}}\leq k-1$ are measurable. 

Then for every bounded measurable function $f\colon \B{G}\to \B{R}$ which is $k$-approximately differentiable almost everywhere, the coefficients $a_{J}$, $|J|_{\mathbb{G}}\leq k$, are measurable
\end{lemma}

\begin{proof}
Assume the hypotheses of the lemma. Fix a bounded measurable function $f\colon \B{G}\to \B{R}$ which is $k$-approximately differentiable almost everywhere with $k$-approximate derivative $\sum_{|J|_{\mathbb{G}}\leq k}a_J(x_0)(x_0^{-1}x)^J$. It is easy to check that $f$ is also $(k-1)$-approximately differentiable almost everywhere with $(k-1)$-approximate derivative $\sum_{|J|_{\mathbb{G}}\leq k-1}a_J(x_0)(x_0^{-1}x)^J$. It then follows from the inductive hypothesis that the coefficients $a_J$, $|J|_{\mathbb{G}}\leq k-1$ are measurable. It suffices to show that $a_{J}|_{B(0,T)}$ is measurable for $|J|_{\mathbb{G}}=k$ for any fixed $T>0$. 

We claim that it suffices to show $a_{J}$, $|J|_{\mathbb{G}}=k$ are measurable when restricted to any compact set $A\subset B(0,T)$ with the following properties:
\begin{itemize}
	\item $a_{J}|_{A}$ are defined at every point of $A$ for $|J|_{\bbG}=k$
	\item $a_{J}|_{A}$ are continuous for all $|J|_{\mathbb{G}}\leq k-1$. 
\end{itemize}
We suppose we can show this implication and see how the general case follows. Since $a_J$ is measurable for each $|J|_{\mathbb{G}}\leq k-1$, we can apply the classical Lusin theorem to $a_{J}|_{B(0,T)}$ for each $|J|_{\bbG}\leq k-1$. For every $m\in \bbN$ we find compact $A_{m}\subset B(0,T)$ such that $\mathcal{L}^{N}(B(0,T)\setminus A_{m})<1/m$ and $a_{J}|_{A_{m}}$ is continuous and everywhere defined for $|J|_{\mathbb{G}}\leq k-1$. Using the assumed implication with $A$ replaced by $A_{m}$, it follows that $a_{J}|_{A_{m}}$ are measurable for every $m$. Since $\mathcal{L}^{N}(B(0,T)\setminus \cup_{m=1}^{\infty}A_m)=0$, it follows easily that $a_{J}|_{B(0,T)}$ is measurable for $|J|_{\bbG}\leq k$  as required. 

Fix $T>0$ and a compact set $A$ as above. Let $B_J\subset\B{R}$ be non-empty closed bounded intervals for each $|J|_{\mathbb{G}}=k$. Let $\varepsilon\in\B{Q}^+$, $R\in\B{Q}^+$, $q_J\in B_J\cap\B{Q}$ for $|J|_{\B{G}}=k$, $r\in(0,R)\cap\B{Q}$. Define for $x_{0}\in A$ and $x\in \bbG\setminus\{x_0\}$,
\[Q(x,x_{0})=\frac{|f(x)-\sum_{|J|_{\mathbb{G}}\leq k-1}a_J(x_0)(x_0^{-1}x)^J-\sum_{|J|_{\mathbb{G}}= k}q_J(x_0^{-1}x)^J|}{d(x,x_0)^k}.\]
To prove measurability is enough to show that $\{x_{0}\in A: a_{J}(x_{0})\in B_{J}\}$ is measurable for all $|J|_{\bbG}=k$. Using Lemma \ref{decomposition}, to do so it suffices to show that for any such $B_{J}, \varepsilon, R, q_{J}, r$, the following set is measurable
\[
\{ x_0\in A: \mathcal{L}^{N} \{ x\in B(x_0,r):Q(x,x_{0})>\varepsilon \} <\varepsilon r^Q \}.
\]
It suffices to show measurability of the function $\phi \colon A\to\B{R}^+$ defined by
\[
\phi(x_0)= \mathcal{L}^{N} \{ x\in B(x_0,r): Q(x,x_{0})>\varepsilon \}.
\]
To prove this, we show that the superlevel sets $U_{\alpha}=\{ x_0\in A: \phi(x_0)>\alpha \}$ are measurable for each $\alpha>0$. It suffices to show that $U_{\alpha}$ is an open subset of $A$. Fix $\alpha>0$ and $x_0\in U_{\alpha}$. For $n\in \bbN$ define the sets
\begin{equation*}
	\begin{split}
	S &= \{ x\in B(x_0,r): Q(x,x_{0})>\varepsilon \},\\
	S_n &= \{ x\in B(x_0,r): Q(x,x_{0})>\varepsilon+ 1/n \}.
	\end{split}
\end{equation*}
Recall that $Q(x,x_0)$ is a measurable function of $x$. Clearly $S_n\subset S_{n+1}$ for each $n\in\B{N}$ and $\cup_{n=1}^{\infty}S_n = S$. Hence $\mathcal{L}^{N}(S_n)\to \mathcal{L}^{N}(S)$. Since $\mathcal{L}^{N}(S)>\alpha$, then there exists $\tilde{n}\in\B{N}$ such that $\mathcal{L}^{N}(S_n)>\alpha$ for $n\geq\tilde{n}$. Define $\tilde{\varepsilon}:=\varepsilon+1/ \tilde{n}$. Observe that there exists $\tilde{\alpha}>\alpha$ such that
\begin{equation}\label{measure_epsilontilde}
\mathcal{L}^{N}\{ x\in B(x_0,r) :Q(x,x_{0})>\tilde{\varepsilon} \} >\tilde{\alpha}.
\end{equation}
Our aim is to show that whenever $y_{0}\in A$ is sufficiently close to $x_{0}$, then
\[\mathcal{L}^{N}\{ x\in B(y_0,r): Q(x,y_{0}) >\varepsilon \} >\alpha.\]
This would show that $\phi(y_{0})>\alpha$ so $y_{0}\in U_{\alpha}$, proving that $U_{\alpha}$ is an open subset of $A$ and completing the proof. To this end, using \eqref{measure_epsilontilde} first fix $0<\Lambda<r$ such that
\[\mathcal{L}^{N}\{ x\in B(x_0,r)\setminus B(x_{0},\Lambda) :Q(x,x_{0})>\tilde{\varepsilon} \} >\tilde{\alpha}.\]
Fix $0<\tilde{\delta}<\Lambda/2$ such that if $y_{0}\in A$ and $d(y_{0},x_{0})<\tilde{\delta}$ we have
\begin{equation}\label{y0}
\mathcal{L}^{N}\{ x\in B(y_0,r)\setminus B(y_{0},\Lambda) :Q(x,x_{0})>\tilde{\varepsilon}\} >\alpha.
\end{equation}

\begin{claim}\label{claimopen}
There exists $0<\delta<\tilde{\delta}$ such that if $y_0\in A$ and $d(y_{0},x_{0})<\delta$, then
\[|Q(x,y_{0})-Q(x,x_{0})|<\tilde{\varepsilon}-\varepsilon \mbox{ for every }x\in B(y_0,r)\setminus B(y_0,\Lambda).\]
In particular, $B(x_{0},\delta)\cap A\subset U_{\alpha}$. Hence $U_{\alpha}$ is open as a subset of $A$.
\end{claim}

\begin{proof}[Proof of Claim \ref{claimopen}]
Define,
\begin{equation*}
\begin{split}
a(x,x_{0})&:=\Big|f(x)-\sum_{|J|_{\mathbb{G}}\leq k-1}a_J(x_0)(x_0^{-1}x)^J-\sum_{|J|_{\mathbb{G}}= k}q_J(x_0^{-1}x)^J \Big|\\
b(x,x_{0})&:=d(x,x_0)^k.
\end{split}
\end{equation*}
Clearly $Q(x,x_{0})=a(x,x_{0})/b(x,x_{0})$. Similarly $Q(x,y_{0})=a(x,y_{0})/b(x,y_{0})$. Notice,
\begin{align*}
|Q(x,x_{0})-Q(x,y_{0})|&=\left|\frac{a(x,x_{0})}{b(x,x_{0})}-\frac{a(x,y_{0})}{b(x,y_{0})}\right|\\
&\leq \frac{|a(x,x_{0})||b(x,x_{0})-b(x,y_{0})|+|b(x,x_{0})||a(x,x_{0})-a(x,y_{0})|}{|b(x,x_{0})||b(x,y_{0})|}.
\end{align*}
If $d(y_{0},x_{0})<\Lambda/2$, then for all $x\in  B(y_0,r)\setminus B(y_0,\Lambda)$ we have 
\[\Lambda^{k}/2^{k}\leq |b(x,x_{0})| \leq (r+\Lambda/2)^{k}\]
and $|b(x,y_{0})| \geq \Lambda^{k}$. Recall that $f$ is bounded and $a_{J}$ is a continuous function on the compact set $A$ for $|J|_{\B{G}}\leq k-1$, hence bounded. Combining these estimates, there exists a constant $C>0$ independent of $x$ and $y_{0}$ such that whenever $d(y_{0},x_{0})<\Lambda/2$, then for all $x\in  B(y_0,r)\setminus B(y_0,\Lambda)$
\[ |Q(x,x_{0})-Q(x,y_{0})| \leq C |a(x,x_{0})-a(x,y_{0})|+C|b(x,x_{0})-b(x,y_{0})| .\]
To conclude the proof, it suffices to show that
\begin{equation}\label{conv1}
\sup_{x\in  B(y_0,r)\setminus B(y_0,\Lambda)} |a(x,x_{0})-a(x,y_{0})| \to 0 \mbox{ as }y_{0}\to x_{0},
\end{equation}
\begin{equation}\label{conv2}
\sup_{x\in  B(y_0,r)\setminus B(y_0,\Lambda)} |b(x,x_{0})-b(x,y_{0})| \to 0 \mbox{ as }y_{0}\to x_{0}.
\end{equation}
Equation \eqref{conv2} follows from Lemma \ref{convergedistance}. To prove \eqref{conv1} we estimate $|a(x,x_{0})-a(x,y_{0})|$ by
\begin{equation*}
	\begin{split}
	&\Bigg| |f(x)-\sum_{|J|_{\mathbb{G}}\leq k-1}a_J(x_0)(x_0^{-1}x)^J-\sum_{|J|_{\mathbb{G}}= k}q_J(x_0^{-1}x)^J|\\
	&\qquad \qquad - |f(x)-\sum_{|J|_{\mathbb{G}}\leq k-1}a_J(y_0)(y_0^{-1}x)^J-\sum_{|J|_{\mathbb{G}}= k}q_J(y_0^{-1}x)^J| \Bigg|\\
	 & \leq\sum_{|J|_{\mathbb{G}}\leq k-1}\underbrace{|a_J(x_0)(x_0^{-1}x)^J-a_J(y_0)(y_0^{-1}x)^J|}_\text{I}\\
	&\qquad \qquad + \sum_{|J|_{\mathbb{G}}=k}\underbrace{q_J| (x_0^{-1}x)^J-(y_0^{-1}x)^J |}_\text{II}.
	\end{split}
\end{equation*}
We estimate the terms of I, for $|J|_{\mathbb{G}}\leq k-1$,
\begin{equation*}
	\begin{split}
		|a_J&(x_0)(x_0^{-1}x)^J-a_J(y_0)(y_0^{-1}x)^J|\\
		&= |a_J(x_0)(x_0^{-1}x)^J-a_J(x_0)(y_0^{-1}x)^J+a_J(x_0)(y_0^{-1}x)^J-a_J(y_0)(y_0^{-1}x)^J|\\
		&\leq |a_J(x_0)||(x_0^{-1}x)^J-(y_0^{-1}x)^J|+|(y_0^{-1}x)^J||a_J(x_0)-a_J(y_0)|.
	\end{split}
\end{equation*}
Notice $|a_J(x_0)|$ is fixed, $|(y_0^{-1}x)^J|$ is bounded over $x\in  B(y_0,r)$, $d(y_{0},x_{0})<\Lambda/2$. That $a_J|_A$ is continuous for $|J|_{\mathbb{G}}\leq k-1$ implies $|a_J(x_0)-a_J(y_0)|\to 0$ as $y_0\to x_0$ with $y_{0}\in A$. Thanks to Lemma \ref{convergeindex} we have that, for $|J|_{\mathbb{G}}\leq k-1$,
\[ \sup_{x\in  B(y_0,r)\setminus B(y_0,\Lambda)}|(x_0^{-1}x)^J-(y_0^{-1}x)^J|\to 0 \mbox{ as }y_{0}\to x_{0}\]
and this suffices to complete the estimate I. The same is true for $|J|_{\B{G}}=k$, which gives the estimate II, hence completing the proof.
\end{proof}

Claim \ref{claimopen} concludes the proof of Lemma \ref{induction} and hence Proposition \ref{diffmeas}.
\end{proof}

\section{Lusin Approximation by $C^{k}$ Functions}\label{Lusindiff}

\begin{theorem*}(Restatement of Theorem \ref{mainthma})
Let $D$ be a measurable subset of $\bbG$ and $f\colon D\to \B{R}$ be measurable. Then the following are equivalent for every non-negative integer $k$:
\begin{enumerate}
\item $f$ is $k$-approximately differentiable at almost every point of $D$. 
\item $f$ admits a Lusin approximation by functions in $C^{k}_{\bbG}(\bbG)$.
\end{enumerate} 
\end{theorem*}

We prove each implication of Theorem \ref{mainthma} in Lemma \ref{mainthma1} and Lemma \ref{mainthma2}. They are based on adapting the Euclidean techniques from \cite{LT94} with the Whitney extension theorem in Carnot groups (Theorem \ref{genclassicalWhitney}) and suitable adaptations to the Carnot group setting using the results of Section \ref{DiGiorgi} and Section \ref{measurability}. The stronger hypothesis of $k$-approximate derivative is used to obtain directly a $C^{k}_{\bbG}(\bbG)$ approximation.

\begin{lemma}\label{mainthma1}
Let $D$ be a measurable subset of $\bbG$ and let $k$ be a non-negative integer. Suppose $f\colon D\to \B{R}$ is a measurable function which is approximately differentiable of order $k$ at almost every point of $D$. Then $f$ admits a Lusin approximation by functions in $C^{k}_{\bbG}(\bbG)$.
\end{lemma}

\begin{proof}
First assume $\mathcal{L}^{N}(D)<\infty$. By replacing $D$ by a subset of full measure, we assume that $f$ has a $k$-approximate derivative at every point of $D$. Denote the $k$-approximate derivative at each point $x\in D$ by  $p(x,y)$. For any multi-index $J$ with $|J|_{\mathbb{G}}\leq k$, define $f_{J}\colon D \to \mathbb{R}$ by $f_{J}(x)~=~X^J p(x,y)|_{y=x}$, where the derivative is taken with respect to $y$. It follows from Proposition \ref{diffmeas} that $f_{J}$ is measurable on $D$ for each $|J|_{\mathbb{G}}\leq k$.

Fix $0<\delta<1$. For each $x\in D$ and $r>0$, define
\[
W_\delta(x,r) = B(x,r)\setminus \{ y\in D : |f(y)-p(x,y)|\leq \delta d(x,y)^k\},
\]
\[
T_{\delta}(r) = \{ (x,y)\in D\times D : d(x,y)<r,\, |f(y)-p(x,y)|>\delta d(x,y)^k \}.
\]
Each set $W_{\delta}(x,r)$ is a measurable subset of $\B{G}$. The map $(x,y)\to p(x,y)$ is continuous in $y$ for each fixed $x$ and measurable in $x$ for each fixed $y$. Hence $(x,y)\to p(x,y)$ is measurable on $D\times D$ \cite[Lemma 4.15]{AB}. This implies $T_\delta(r)$ is a measurable subset of $\B{G}\times\B{G}$. Let $Z=\{(x,y)\in \bbG\times \bbG: d(x,y)<r \mbox{ and }y\notin D\}$, also a measurable subset of $\bbG\times \bbG$. Then $W_{\delta}(x,r) = \{ y\in \bbG : (x,y)\in T_{\delta}(r) \cup Z \}$. Since $T_{\delta}(r) \cup Z$ is a measurable subset of $\bbG\times \bbG$, Fubini's theorem implies that $\mathcal{L}^N(W_{\delta}(x,r))$ is measurable as a function of $x\in D$ for any $r>0$. 

Recall that $V$ is the measure of the unit ball. Define,
\begin{align*}
A_j&=\Big\{x\in D :\mathcal{L}^N(W_{\delta}(x,r)) \leq Vr^Q/2^{Q+2} \mbox{ for all } r\leq 1/j\Big\},\\
B_j&= \Big\{x\in D : |f_{J}(x)|\leq j \text{ for all multi indices } |J|_{\B{G}}\leq k\Big\},\\
C_{j}&=A_{j}\cap B_{j}.
\end{align*}
Observe that $\mathcal{L}^N(W_{\delta}(x,r))$ is an increasing continuous function of $r$ for each fixed $x$. Fix a countable dense subsequence $\{r_i\}_{i=1}^{\infty} \subset(0,1/j]$. Then
\[
A_j = \bigcap_{i=1}^{\infty}\Big\{x\in D : \mathcal{L}^N(W_{\delta}(x,r_{i})) \leq Vr_i^Q / 2^{Q+2}\Big\}.
\]
This shows $A_{j}$ is a countable intersection of measurable sets, hence measurable. Clearly $B_{j}$ and hence $C_{j}$ are also measurable since the $f_{J}$ are measurable.

Temporarily fix $j\in \bbN$ and $x,y\in C_j$ with $0<d(x,y)\leq 1/j$. Let $r=d(x,y)$ and
\[
S(x,y,r,j) := [B(x,r)\cap B(y,r)]\setminus [W_{\delta}(x,r)\cup W_{\delta}(y,r)].
\]
By Lemma \ref{lemma_R}, $\mathcal{L}^N(B(x,r)\cap B(y,r)) \geq Vr^{Q}/2^{Q}$. By assuming $x,y\in C_{j}\subset A_j$, we have $\mathcal{L}^N(W_\delta(x,r))\leq Vr^{Q}/2^{Q+2}$ and $\mathcal{L}^N(W_{\delta}(y,r)))\leq Vr^{Q}/2^{Q+2}$. Hence
\begin{equation}\label{eq_measureofS0}
 \mathcal{L}^N(S(x,y,r,j)) \geq Vr^Q/2^{Q+1}.
\end{equation}
Given any $z\in S(x,y,r,j)$, define the polynomial $q(z) := p(y,z)-p(x,z)$. Then
\begin{align*}
|q(z)| &\leq |p(y,z)-f(z)|+|f(z)-p(x,z)|\\
&\leq \delta(d(y,z)^k+d(x,z)^k)\\
&\leq 2\delta r^k.
\end{align*}
Using \eqref{eq_measureofS0}, we can apply Lemma \ref{degiorgi_carnot} with $E=S(x,y,r,j)\subset B(y,r)$, $A=V/2^{Q+1}$, and $P$ replaced by $q$. This gives
\begin{align}\label{estimate_X^I0}
|(X^{I}q)(y)|&=|f_I(y)-(X^Ip)(x,y)|  \\
&\leq \frac{C}{r^{Q+|I|_{\B{G}}}}\int_{S(x,y,r,j)} |q(z)|dz\nonumber \\
&\leq C\delta r^{k-|I|_{\B{G}}}. \nonumber
\end{align}
The constants $C$ above vary on each line, but both depend only on $k$ and $Q$.

The sets $C_j$ are increasing as $j\to\infty$
and $\mathcal{L}^N(D\setminus \bigcup_j C_j)=0$ because $f$ has $k$-derivative $p(x,y)$ at every $x\in D$ and almost every point of $D$ is a density point of $D$. Since $\mathcal{L}^{N}(D)<\infty$, for any $\varepsilon >0$ there exists  $j_0\in \bbN$ such that $\mathcal{L}^N(D\setminus C_{j_0})\leq \varepsilon/2$. We can then choose a closed subset $F\subset C_{j_0}$ such that $\mathcal{L}^N(D\setminus F)<\varepsilon$. Combining this with \eqref{estimate_X^I0} and the definition of $B_{j}$, we have shown the following. For any $\varepsilon>0$ and $\delta>0$, there exists a closed set $F(\varepsilon, \delta)\subset D$ with $\mathcal{L}^N(D\setminus F(\varepsilon,\delta))<\varepsilon$  and $N(\varepsilon, \delta)\in \bbN$ such that for $x,y\in F(\varepsilon, \delta)$,
\begin{equation}\label{sskp0}
|f_I(y)-(X^Ip)(x,y)|\leq C\delta d(x,y)^{k-|I|_{\B{G}}} \mbox{ for }|I|_{\bbG}\leq k,\, d(x,y)<1/N(\varepsilon, \delta)
\end{equation}
\begin{equation}\label{bounded0}
|f_{J}(x)|\leq N(\varepsilon, \delta) \text{ for } |J|_{\B{G}}\leq k.
\end{equation}
Fix $\varepsilon > 0$ and define $F=\cap_{m=1}^{\infty}F(\varepsilon/2^m, 1/m)$. Clearly $\mathcal{L}^N(D\setminus F)<\varepsilon$ and $F$ is a closed set.

We now verify the conditions of Theorem \ref{genclassicalWhitney} for $\{f_I\}_{|I|_{\B{G}}\leq k}$ restricted to $F$. Since $F\subset F(\varepsilon/2, 1)$, \eqref{bounded0} gives $|f_{J}(x)|\leq N(\varepsilon/2, 1)$ for all $x\in F$ and multi indices $|J|_{\bbG}\leq k$. Fix $\eta > 0$ and $\bar{x}\in F$. Choose $M\in \bbN$ so that $1/M<\eta$. Suppose $x,y\in F$ with $d(\bar{x},x),d(\bar{x},y)<1/(2N(\varepsilon/2^{M},1/M))$. It follows $x,y\in F(\varepsilon/2^{M},1/M)$ with $d(x,y)<1/(N(\varepsilon/2^{M},1/M))$. Then by \eqref{sskp0}, recalling $1/M<\eta$, we have 
\[|f_I(y)-(X^Ip)(x,y)|\leq C\eta d(x,y)^{k-|I|_{\B{G}}} \mbox{ for }|I|_{\bbG}\leq k.\]
Hence we can can apply Theorem \ref{genclassicalWhitney} to $\{f_I\}_{|I|_{\B{G}}\leq k}$ on $F$. This yields a $C^{k}_{\bbG}(\bbG)$ function $g\colon \bbG \to \bbR$ which extends $\{f_I\}_{|I|_{\B{G}}\leq k}$ from $F$. Since $\mathcal{L}^N(D\setminus F)<\varepsilon$, this provides the required $C^k_{\bbG}(\bbG)$ approximation of $f$. This proves the lemma in the case $\mathcal{L}^{N}(D)<\infty$.

If $\mathcal{L}^N(D)=\infty$ we write $\bbG$ as a union of annuli (sets of the form $B(0,R)\setminus B(0,S)$) which are alternately fat ($R-S$ is relatively large) and thin ($R-S$ is relatively small). On each fat annuli we use the previous case to approximate $f$ in the Lusin sense on that annulus by a $C^{k}_{\bbG}$ function. In the thin annuli we use Theorem \ref{genclassicalWhitney} to interpolate and hence combine the individual $C^{k}_{\bbG}$ functions into one $C^{k}_{\bbG}$ function on $\bbG$. Provided the individual approximations are sufficiently strong in the Lusin sense and the union of the thin annuli has small measure, this yields the required Lusin approximation of $f$.

\end{proof}

\begin{lemma}\label{mainthma2}
Let $D$ be a measurable subset of $\bbG$ and let $k$ be a non-negative integer. Suppose $f\colon D\to \B{R}$ is a measurable function with the Lusin property of order $k$. Then $f$ is approximately differentiable of order $k$ at almost every point of $D$. 
\end{lemma}

\begin{proof}
Suppose $f$ has the Lusin property of order $k$ on $D$. Then for almost every $x\in D$, there exists $u\in C^k_{\mathbb{G}}(\mathbb{G})$ such that $\{z\in D: f(z)=u(z)\}$ contains $x$ and has density one at $x$. We claim for that such a point $x$
\begin{equation*}
	\aplim_{y\to x}\frac{|f(y)-P_k(u,x,y)|}{d(x,y)^k}=0.
\end{equation*}
To see this, first notice
\begin{align*}
	\aplim_{y\to x}\frac{|f(y)-P_k(u,x,y)|}{d(x,y)^k}\leq& \aplim_{y\to x}\frac{|f(y)-u(y)|}{d(x,y)^k}+\aplim_{y\to x}\frac{|u(y)-P_k(u,x,y)|}{d(x,y)^k}.
\end{align*}
Since $x\in \{z\in D: f(z)=u(z)\}$ is a point of density one,
\begin{equation*}
	\aplim_{y\to x}\frac{|f(y)-u(y)|}{d(x,y)^k}=0.
\end{equation*}
On the other hand, Theorem \ref{Teo_Taylor} implies
\begin{equation*}
	\aplim_{y\to x}\frac{|u(y)-P_k(u,x,y)|}{d(x,y)^k}=0.
\end{equation*}
\end{proof}

Taken together, Lemma \ref{mainthma1} and Lemma \ref{mainthma2} prove Theorem \ref{mainthma}.

\section{Lusin Approximation by $\mathrm{Lip}(k, \bbG)$ Functions}\label{Lusintaylor}

\begin{theorem*}(Restatement of Theorem \ref{mainthmb})
Let $D$ be a measurable subset of $\bbG$ with $\mathcal{L}^{N}(D)<\infty$. Let $f\colon D\to \B{R}$ be measurable. Then the following are equivalent for every positive integer $k$:
\begin{enumerate}
\item $f$ has an approximate $(k-1)$-Taylor polynomial at almost every point of~$D$.
\item $f$ admits a Lusin approximation on $D$ by functions in $\mathrm{Lip}(k, \bbG)$.
\end{enumerate} 
\end{theorem*}

We prove each implication of the equivalence in Theorem \ref{mainthmb} separately, in Lemma \ref{mainthmb1} and Lemma \ref{mainthmb2}. The proof of Lemma \ref{mainthmb1} is similar to that of Lemma \ref{mainthma1} with minor differences, since the weaker hypothesis leads naturally to the weaker conclusion.

\begin{lemma}\label{mainthmb1}
Let $D$ be a measurable subset of $\bbG$ with $\mathcal{L}^{N}(D)<\infty$. Let $f\colon D\to\B{R}$ be measurable. Let $k$ be a positive integer. Suppose $f$ has an approximate $(k-1)$-Taylor polynomial $p(x_0,x)$ at almost every point $x_0\in D$. Then $f$ admits a Lusin approximation on $D$ by functions in $\mathrm{Lip}(k, \bbG)$
\end{lemma}

\begin{proof}
Since $\mathcal{L}^{N}(D)<\infty$, for any $\varepsilon >0$ there is $R>0$ so that $\mathcal{L}^{N}(D\setminus B(0,R))<\varepsilon$. Hence, by replacing $D$ by $D\cap B(0,R)$, it suffices to prove the lemma for bounded $D$. By replacing $D$ by a smaller set of full measure, we assume also that $f$ has an approximate $(k-1)$-Taylor polynomial $p(x,y)$ at every point of $x\in D$. For any multi-index $J$ with $|J|_{\mathbb{G}}\leq k-1$, define $f_{J}\colon D \to \mathbb{R}$ by $f_{J}(x)~=~X^J p(x,y)|_{y=x}$. It follows from Corollary \ref{taylormeas} that $f_J$ are measurable. 

For each $x\in D$, $j\in \bbN$, $r>0$, define
\[
W_j(x,r) = B(x,r)\setminus\{ y\in D : |f(y)-p(x,y)|\leq j d(x,y)^k\},
\]
\[
T_j(r) = \{ (x,y)\in D\times D : d(x,y)<r,\, |f(y)-p(x,y)|>j d(x,y)^k \}.
\]
Each $W_j(x,r)$ is a measurable subset of $\B{G}$. The map $(x,y)\to p(x,y)$ is continuous in $y$ and measurable in $x$. Hence it is a measurable function on $D\times D$ \cite[Lemma 4.15]{AB}. It follows that each set $T_j(r)$ is a measurable subset of $\B{G}\times\B{G}$.  Let $Z=\{(x,y)\in \bbG\times \bbG: d(x,y)<r \mbox{ and }y\notin D\}$, also a measurable subset of $\bbG\times \bbG$. Then $W_{j}(x,r) = \{ y\in \bbG : (x,y)\in T_{j}(r) \cup Z \}$. Since $T_{j}(r) \cup Z$ is a measurable subset of $\bbG\times \bbG$, Fubini's theorem implies that $\mathcal{L}^N(W_{j}(x,r))$ is measurable as a function of $x\in D$ for any $r>0$. Define
\begin{align*}
A_j&=\Big\{x\in D : \mathcal{L}^N(W_j(x,r)) \leq Vr^Q / 2^{Q+2} \mbox{ for all } r\leq 1/j\Big\},\\
B_j&= \Big\{x\in D : |f_{J}(x)|\leq j \text{ for all multi indices } |J|_{\B{G}}\leq k-1\Big\},\\
C_j&= A_j \cap B_j.
\end{align*}
As in the proof of Lemma \ref{mainthma1}, the sets $A_{j}, B_{j}, C_{j}$ are measurable for $j\in \bbN$.

Temporarily fix $j\in \bbN$ and $x,y\in C_j$ with $0<d(x,y)\leq 1/j$. Let $r=d(x,y)$ and
\[
S(x,y,r,j) = [B(x,r)\cap B(y,r)]\setminus [W_j(x,r)\cup W_j(y,r)].
\]
By Lemma \ref{lemma_R}, $\mathcal{L}^N(B(x,r)\cap B(y,r)) \geq Vr^{Q}/2^{Q}$. By the assumption $x,y\in C_{j}$, $\mathcal{L}^N(W_j(x,r))\leq Vr^{Q}/2^{Q+2}$ and $\mathcal{L}^N(W_j(y,r)))\leq Vr^{Q}/2^{Q+2}$. It follows that $\mathcal{L}^N(S(x,y,r,j)) \geq Vr^Q/2^{Q+1}$.  For $z\in S(x,y,r,j)$ let $q(z) = p(y,z)-p(x,z)$. We have
\begin{align*}
|q(z)| &\leq |p(y,z)-f(z)|+|f(z)-p(x,z)|\\
&\leq j(d(y,z)^k+d(x,z)^k)\\
&\leq 2jr^k.
\end{align*}
We apply Lemma \ref{degiorgi_carnot} with  $E=S(x,y,r,j)\subset B(y,r)$, $A=V/2^{Q+1}$, and $P$ replaced by $q$. This gives, for all $|I|_{\B{G}}\leq k-1$,
\begin{align}\label{estimate_X^I}
|(X^{I}q)(y)|&=|f_I(y)-X^Ip(x,y)|\\
&\leq \frac{C}{r^{Q+|I|_{\B{G}}}}\int_{S(x,y,j)} |q(z)|dz \nonumber \\
&\leq Cjr^{k-|I|_{\B{G}}}.\nonumber 
\end{align}
The constant $C$ varies in the two lines but depends only on $k$ and $Q$.

The sets $C_j$ are increasing as $j\to\infty$
and $\mathcal{L}^N(D\setminus \bigcup_j C_j)=0$ as $f$ has approximate $(k-1)$-Taylor polynomial $p(x,y)$ for almost every $x$. Since $\mathcal{L}^{N}(D)<\infty$, given $\varepsilon >0$ we can choose $j_0$ such that $\mathcal{L}^N(D\setminus C_{j_0})\leq \varepsilon / 2$ and then choose a closed subset $F\subset C_{j_0}$ such that $\mathcal{L}^N(D\setminus F)<\varepsilon$. 

For $x,y\in F$, \eqref{estimate_X^I} gives $|f_I(y)-X^Ip(x,y)|\leq Cj_0d(x,y)^{k-|I|_{\B{G}}}$ for $x,y\in F$ with $d(x,y)\leq 1/j_{0}$ and $|I|_{\bbG}\leq k-1$. For $d(x,y)\geq 1/j_{0}$ the right side is bounded below with a constant independent of $x,y$. Meanwhile the left side is bounded above independently of $x,y$. Indeed, we have $|f_{I}(x)|\leq j_{0}$ for all $x\in F$ and $|I|_{\bbG}\leq k-1$. At the same time, Proposition \ref{dxa} and \eqref{def_a_J} imply that the coefficients of $p(x,y)$ are a linear combination of $f_{I}(x)$. Hence, since $F\subset D$ and $D$ is bounded, $|X^Ip(x,y)|$ is bounded for $x,y\in F$. Consequently for all $x,y\in F$,
\begin{equation}\label{sskp}
|f_I(y)-X^Ip(x,y)|\leq Cj_0d(x,y)^{k-|I|_{\B{G}}}\ \text{ for } |I|_{\B{G}}\leq k-1.
\end{equation}
We also have $|f_I(x)|\leq j_0$ for any $x\in F$ and $|I|_{\B{G}}\leq k-1$. Combining this with \eqref{sskp}, this shows the collection $\{f_I\}_{|I|_{\B{G}}\leq k-1}$ restricted to $F$ belong to the set $ \mathrm{Lip}(k, F)$, where $k$ is replaced by $k-1$ and $\gamma$ by $k$. We then apply Theorem \ref{extension} to get a Lusin approximation of $f$ on $F$ hence on $D$, i.e. $g\in \mathrm{Lip}(k, \bbG)$ such that $X^{J}g=f_{J}$ on $F$ for $|J|_{\bbG}\leq k-1$. 
\end{proof}

\begin{lemma}\label{mainthmb2}
Suppose a measurable function $f\colon D \to \bbR$ admits Lusin approximation on $D$ by functions in $\mathrm{Lip}(k, \bbG)$ for some positive integer $k$. Then $f$ has an approximate $(k-1)$-Taylor polynomial at almost every point of~$D$.
\end{lemma}

\begin{proof}
Fix $\varepsilon >0$. Using our hypothesis, choose $g\in \mathrm{Lip}(k, \bbG)$ and a measurable set $A\subset D$ such that $\mathcal{L}^{N}(D\setminus A)<\varepsilon$ and $g|_{A}=f_{A}$. Applying the definition of $\mathrm{Lip}(k, \bbG)$ with a multi-index of length $0$, there exists a constant $M$ and for every $x_{0}\in \bbG$ a polynomial $P(x_{0},x)$ of degree at most $k-1$ such that
\[|g(x)-P(x_{0},x)|\leq Md(x,x_{0})^{k} \mbox{ for all }x\in \bbG.\]
This implies that $g$ has $(k-1)$-approximate Taylor polynomial $P(x_{0},x)$ at $x_{0}$, i.e.
\[ \aplimsup_{x\to x_{0}} \frac{|g(x)-P(x_{0},x)|}{d(x,x_{0})^{k}}<\infty.\]
If $x_{0}\in A$ is a density point of $A$, it then follows that
\[ \aplimsup_{x\to x_{0}} \frac{|f(x)-P(x_{0},x)|}{d(x,x_{0})^{k}}<\infty.\]
Hence $f$ has an approximate $(k-1)$-Taylor polynomial at almost every point of $A$. Since we could choose $A$ with $\mathcal{L}^N(D\setminus A)<\varepsilon$ for any fixed $\varepsilon >0$, it follows that $f$ has an approximate $(k-1)$-Taylor polynomial at almost every point of $D$.  
\end{proof}

Taken together, Lemma \ref{mainthmb1} and Lemma \ref{mainthmb2} prove Theorem \ref{mainthmb}.

\begin{lemma}\label{mainthmb3}
Suppose $f\in \mathrm{Lip}(k, \bbG)$. Then $f\in C^{k-1}_{\bbG}(\bbG)$, $X^{J}f$ is bounded for $|J|_{\bbG}\leq k-1$, and $X^{J}f$ is Lipschitz for $|J|_{\bbG}= k-1$.
\end{lemma}

\begin{proof}
Suppose $f\in \mathrm{Lip}(k, \bbG)$. This means that Definition \ref{lipG} holds with $k$ replaced by $k-1$ and $\gamma$ replaced by $k$. Hence there exists a constant $M$ and for every $x_{0}\in \bbG$ a polynomial $P(x_{0},x)$ of homogeneous degree at most $k-1$ such that for all multi-indices $|J|_{\bbG}\leq k-1$, $x, x_{0}\in \bbG$,
\[|X^{J}f(x_{0})|\leq M,\]
\begin{equation}\label{lipbound}
|(X^{J}f)(x)-X^{J}P(x_{0},x)|\leq Md(x,x_{0})^{k-|J|_{\bbG}}.
\end{equation}
Clearly this implies $X^{J}f$ is bounded for $|J|_{\bbG}\leq k-1$. Let $J$ be a multi-index with $|J|_{\bbG}=k-1$. Then $X^{J}P(x_{0},x)$ has homogeneous degree at most $0$, so is constant in $x$. Setting $x=x_{0}$ in \eqref{lipbound} shows $X^{J}P(x_{0},x_{0})=(X^{J}f)(x_{0})$. Hence $X^{J}P(x_{0},x)=(X^{J}f)(x_{0})$ for every $x\in \bbG$. Substituting this back into \eqref{lipbound} and using $|J|_{\bbG}\leq k-1$ gives for every $x,x_{0}\in \bbG$,
\[ |(X^{J}f)(x)-(X^{J}f)(x_{0}) |\leq Md(x,x_{0}). \]
This proves $X^{J}f$ is Lipschitz for every multi-index $J$ with $|J|_{\bbG}=k-1$, as required.
\end{proof}

\end{document}